\documentclass[a4paper,12pt]{amsart}

\usepackage{amsmath}
\usepackage{amsfonts}	
\usepackage{amssymb}  	
\usepackage{mathrsfs}
\usepackage{mathtools}	
\usepackage{savesym}
\usepackage{graphicx}  	
\usepackage{caption}  	
\usepackage{float} 	  	

\usepackage[shortlabels,inline]{enumitem} 

\usepackage[a4paper,%
			margin=1in,%
			top=1.2in,
			bottom=1.2in,
			]{geometry}

\usepackage[numbers,sort]{natbib} 	

\usepackage{aliascnt} 
\usepackage[pdfusetitle,
		bookmarks=true,
		pagebackref=false,
		colorlinks,
		linkcolor=blue,
		citecolor=blue,
		urlcolor=black]{hyperref}

\usepackage{amsthm}

\theoremstyle{plain}
\newaliascnt{theorem}{dummy}
\newtheorem{theorem}[theorem]{Theorem}
\aliascntresetthe{theorem}

\newaliascnt{proposition}{dummy}
\newtheorem{proposition}[proposition]{Proposition}
\aliascntresetthe{proposition}

\newaliascnt{corollary}{dummy}
\newtheorem{corollary}[corollary]{Corollary}
\aliascntresetthe{corollary}

\newaliascnt{lemma}{dummy}
\newtheorem{lemma}[lemma]{Lemma}
\aliascntresetthe{lemma}

\newaliascnt{conjecture}{dummy}

\aliascntresetthe{conjecture}

\theoremstyle{definition}
\newaliascnt{definition}{dummy}
\newtheorem{definition}[definition]{Definition}
\aliascntresetthe{definition}

\newaliascnt{example}{dummy}

\aliascntresetthe{example}

\theoremstyle{definition}
\newaliascnt{setting}{dummy}
\newtheorem{setting}[setting]{Setting}
\aliascntresetthe{setting}

\newaliascnt{remark}{dummy}
\newtheorem{remark}[remark]{Remark}
\aliascntresetthe{remark}

\numberwithin{equation}{section} 

\allowdisplaybreaks

\newcommand{\calF}{\mathcal{F}}

\newcommand{\calP}{\mathcal{P}}

\newcommand{\calD}{\mathcal{D}}

\newcommand{\calL}{\mathcal{L}}

\newcommand{\scrB}{\mathscr{B}}

\newcommand{\bbR}{\mathbb{R}}

\newcommand{\bbZ}{\mathbb{Z}}
\newcommand{\bbN}{\mathbb{N}}
\newcommand{\bbP}{\mathbb{P}}

\DeclareMathOperator*{\intersect}{\cap}

\DeclareMathOperator*{\Union}{\bigcup}

\providecommand{\abs}[1]{\lvert#1\rvert}
\providecommand{\Abs}[1]{\left\lvert#1\right\rvert}
\providecommand{\norm}[2][]{\lVert#2\rVert_{#1}}

\newcommand{\euclid}[1][d]{\mathbb{R}^{#1}}

\newcommand{\dimH}{\dim_{\mathrm{H}}}

\newcommand{\uDim}[1]{\overline{\dim}_{\mathrm{#1}}}
\newcommand{\lDim}[1]{\underline{\dim}_{\mathrm{#1}}}

\newcommand{\kerZ}[2]{Z_{#1}(#2)}
\newcommand{\kerTheta}[3][s]{J_{\theta,#2}^{#1}(#3)}
\newcommand{\capTheta}[3][s]{C_{\theta,#2}^{#1}(#3)}
\newcommand{\sumTheta}[3][s]{S_{\theta,#2}^{#1}(#3)}
\newcommand{\kerPhi}[3][s]{J_{\Phi,#2}^{#1}(#3)}
\newcommand{\capPhi}[3][s]{C_{\Phi,#2}^{#1}(#3)}
\newcommand{\sumPhi}[3][s]{S_{\Phi,#2}^{#1}(#3)}
\newcommand{\kerProf}[4]{J_{\Phi,#1}^{#2,#3}(#4)}
\newcommand{\capProf}[4]{C_{\Phi,#1}^{#2,#3}(#4)}
\newcommand{\capProfAlpha}[5][\alpha]{C_{\Phi_{#1}, #2^{1/#1}}^{#1 #3, #1 #4}(#5)}
\newcommand{\kerProfAlpha}[5][\alpha]{J_{\Phi_{#1}, #2^{1/#1}}^{#1 #3, #1 #4}(#5)}
\newcommand{\sumProfile}[3]{S_{\Phi,#1}^{#2}(#3)}
\newcommand{\kerAlphaPhi}[4]{J_{\Phi_{\alpha},#1}^{#2,#3}(#4)}
\newcommand{\capAlphaPhi}[4]{C_{\Phi_{\alpha},#1}^{#2,#3}(#4)}
\newcommand{\kerBetaPhiProf}[5][\beta]{J_{\Phi_{#1}, #2}^{#3,#4}(#5)}
\newcommand{\kerPsiPhi}[3]{\psi_{\Phi,#1}^{#2}(#3)}
\newcommand{\kerPsiTheta}[3]{\psi_{\theta,#1}^{#2}(#3)}
\newcommand{\kerGeoCount}[3]{\phi_{#1}^{#2}(#3)}
\newcommand{\xWy}{x\wedge y}
\newcommand{\xyAbs}{\abs{x-y}}

\newcommand{\ol}[1]{\overline{#1}}
\newcommand{\bfa}{\mathbf{a}}
\newcommand{\pia}{\pi^{\mathbf{a}}}
\newcommand{\DeltaPiXY}{\abs{\pia(x)-\pia(y)}}
\newcommand{\piV}{P_{V}}
\newcommand{\xyPiV}{\abs{\piV x - \piV y}}
\newcommand{\Balph}[1]{B_{\alpha}(#1)}
\newcommand{\xyBalpha}{\abs{B_\alpha(x) - B_\alpha(y)}}

\newcommand{\tPhiDim}[1][\Phi]{$ #1 $\nobreakdash-intermediate dimension}
\newcommand{\tThetaDim}[1][\theta]{$ #1 $\nobreakdash-intermediate dimension}
\newcommand{\tPhiDims}[1][\Phi]{$ #1 $\nobreakdash-intermediate dimensions}
\newcommand{\tThetaDims}[1][\theta]{$ #1 $\nobreakdash-intermediate dimensions}
\newcommand{\mFor}{\quad\text{for }}
\newcommand{\mAnd}{\quad\text{ and }\quad}

\title{Intermediate dimensions under self-affine codings}

\author{Zhou Feng}
\address{Department of Mathematics\\
	The Chinese University of Hong Kong\\
	Shatin,  Hong Kong}
\curraddr{}
\email{\href{mailto: zfeng@math.cuhk.edu.hk}{zfeng@math.cuhk.edu.hk}}
\thanks{}


\date{}
\dedicatory{}
\thanks{}

\begin{document}
\begin{abstract}
	Intermediate dimensions were recently introduced by Falconer, Fraser, and Kempton [Math.~Z.~296~(2020)] to interpolate between the Hausdorff and box-counting dimensions. In this paper, we show that for every subset $ E $ of the symbolic space, the intermediate dimensions of the projections of $ E $ under typical self\nobreakdash-affine coding maps are constant and given by formulas in terms of capacities. Moreover, we extend the results to the generalized intermediate dimensions in several settings, including the orthogonal projections in Euclidean spaces and the images of fractional Brownian motions.
\end{abstract}

\maketitle

\section{Introduction}

The study on the dimensions of projections of sets has a long history. For a survey of this topic, please refer to \cite{FalconerEtAl2015}. In this paper, we focus on the intermediate dimensions of projections of sets under the coding maps associated with typical affine iterated function systems.

In what follows, we fix a family of $ d\times d $ invertible real matrices $ T_{1},\ldots, T_{m} $ with $ \norm{T_{j}} < 1 $ for $ 1 \leq j \leq m $. Let $ \mathbf{a} = (a_{1},\ldots, a_{m}) \in \euclid[dm] $. By an \textit{affine iterated function system} (affine IFS) we mean a finite family $ \calF^{\bfa} = \{ f_{j}^{\bfa}\}_{j=1}^{m} $ of affine maps taking the form
\begin{equation*}
	f_{j}^{\bfa}(x) = T_{j} x + a_{j} \mFor 1 \leq j \leq m.
\end{equation*}
Here we write $ f_{j}^{\mathbf{a}} $ instead of $ f_{j} $ to emphasize its dependence on $ \bfa $. It is well known \cite{Hutchinson1981} that there exists a unique non-empty compact set $ K^{\bfa} $ such that
\begin{equation*}
	K^{\bfa} = \Union_{j=1}^{m} f_{j}^{\bfa}(K^{\bfa}).
\end{equation*}
We call $ K^{\bfa} $ the \textit{self-affine set} generated by $ \calF^{\bfa} $. Write $ \Sigma := \{1, \ldots, m\}^{\bbN} $. The (self-affine) \textit{coding map} $ \pia \colon \Sigma \to \euclid $ associated with $ \calF^{\bfa} $ is
\begin{equation}\label{eq:def-coding}
	\pi^{\mathbf{a}}(\mathbf{i}) := \lim_{n\to\infty} f_{i_{1}}^{\mathbf{a}}\circ \cdots \circ f_{i_{n}}^{\mathbf{a}} (0) \quad \text{for } \mathbf{i} = i_{1}\ldots i_{n} \ldots \in \Sigma.
\end{equation}
It is well known \cite{Hutchinson1981} that $ K^{\bfa} = \pia(\Sigma) $.

There is an amount of work studying various dimensional properties of projected sets and measures under typical coding maps \cite{Falconer1988,Solomyak1998,Kaeenmaeki2004,JordanEtAl2007,KaeenmaekiVilppolainen2010,JaervenpaeaeEtAl2014,FengEtAl2022}. Let $ \calL^{d} $ denote the Lebesgue measure on $ \euclid[d] $. In his seminal paper~\cite{Falconer1988}, Falconer showed that the Hausdorff and box-counting dimensions of self-affine sets $ K^{\bfa} = \pia(\Sigma) $ remain as a common constant for $ \calL^{dm} $-a.e.~$ \bfa $ provided that $ \norm{T_{j}} < 1/3 $ for all $ j $. The upper bound in this norm condition was later relaxed to $ 1/2 $ by Solomyak~\cite{Solomyak1998}. Assuming $ \norm{T_{j}} < 1/2 $ for all $ j $, very recently Feng, Lo, and Ma~\cite{FengEtAl2022} showed that for every Borel set $ E \subset \Sigma $, each of the Hausdorff, packing, upper, and lower box-counting dimensions of $ \pia(E) $ is constant for $ \calL^{dm} $-a.e.~$ \bfa $. In this paper, letting $ E \subset \Sigma $, we obtain an analogous constancy result about the intermediate dimensions of $ \pia(E) $ for $ \calL^{dm} $-a.e.\ $ \bfa $.

Intermediate dimensions were introduced by Falconer, Fraser, and Kempton~\cite{FalconerEtAl2020} to interpolate between the Hausdorff and box-counting dimensions, see \cite{Falconer2021a} for a survey. Despite their extremely recent introduction, the intermediate dimensions have already seen interesting applications, for example \cite[Section 6]{BurrellEtAl2021} and \cite{BanajiKolossvary2021}. To avoid problems of definition, throughout the paper we assume that all the sets, whose dimensions are considered, are non-empty and bounded. Denote the diameter of a set $ U \subset \euclid $ by $ \abs{U} $.

\begin{definition}\label{def:thetaDimByCover}
	Let $ F \subset \euclid $. For $ 0 \leq
		\theta \leq 1 $, the \textit{upper \tThetaDim} of $ F $ is defined by
	\begin{align*}
		\uDim{\theta} & F = \inf \{ s \geq 0\colon \text{ for all } \varepsilon > 0, \text{ there exists } r_{0} \in (0, 1] \text{ such that for all } r \in (0, r_{0}),                                            \\
		              & \text{there exists a cover } \{U_{i} \} \text{ of } F \text{ such that }  r^{1/\theta} \leq \abs{U_{i}} \leq r \text{ for all } i \text{ and } \sum_{i} \abs{U_{i}}^{s} \leq \varepsilon \}
	\end{align*}
	and the \textit{lower \tThetaDim} of $ F $ is defined by
	\begin{align*}
		\lDim{\theta} F = \inf \{ & s \geq 0\colon \text{ for all } \varepsilon > 0  \text{ and } r_{0} \in (0, 1], \text{ there exists } r \in (0, r_{0}) \text{  and }                                             \\
		                          & \text{ a cover } \{U_{i} \} \text{ of } F \text{ such that }  r^{1/\theta} \leq \abs{U_{i}} \leq r \text{ for all } i \text{ and } \sum_{i} \abs{U_{i}}^{s} \leq \varepsilon \}.
	\end{align*}
\end{definition}
It is immediate that the Hausdorff dimension $ \dimH F $, the upper box-counting dimension $ \uDim{B} F$, and the lower box-counting dimension $ \lDim{B} F $ are the extreme cases of the \tThetaDims. Specifically,
\begin{equation*}
	\dimH F = \lDim{0} F = \uDim{0} F,\quad \uDim{B} = \uDim{1} F, \quad \text{and}\quad \lDim{B} F = \lDim{1} F.
\end{equation*}

Below we state our first main result on the \tThetaDims\ of $ \pia(E) $ for $ E \subset \Sigma $ in terms of the \textit{capacity dimensions} $ \lDim{C,\theta}E$, $ \uDim{C,\theta}E $ whose rigorous definitions are given in \autoref{def:theta-CapDim}.

\begin{theorem}\label{thm:main}
	Let $ 0 < \theta \leq 1 $ and $ E \subset \Sigma $. Then the followings hold.
	\begin{enumerate}[(i)]
		\item \label{itm:UB} For all $ \mathbf{a} \in \euclid[md] $,
		      \begin{equation*}
			      \lDim{\theta} \pia (E) \leq \lDim{C,\theta} E \quad\text{  and }\quad  \uDim{\theta} \pia (E) \leq \uDim{C,\theta} E .
		      \end{equation*}
		\item \label{itm:LB} Assume $ \norm{T_{j}} < 1/2 $ for $ 1 \leq j \leq m $. Then for $ \calL^{dm} $-a.e.\ $ \mathbf{a} \in \euclid[dm] $,
		      \begin{equation*}
			      \lDim{\theta} \pia (E) = \lDim{C,\theta} E \quad\text{  and }\quad  \uDim{\theta} \pia (E) = \uDim{C,\theta} E.
		      \end{equation*}
	\end{enumerate}
\end{theorem}

\autoref{thm:main} is proved through a capacity approach by adapting and extending some ideas in \cite{Falconer2021,BurrellEtAl2021,FengEtAl2022}. Our definitions of kernels are inspired by, but different from that of Burrell, Falconer, and Fraser~\cite{BurrellEtAl2021} where the projection theorems are established for the \tThetaDims\ under the orthogonal projections in Euclidean spaces. It is these new kernels that reveal a unified computational scheme and pave the way for the extensions to the generalized intermediate dimensions.

In \cite{Banaji2020}, Banaji generalized the \tThetaDims\ to the so-called \textit{\tPhiDims} $ \lDim{\Phi} F $, $ \uDim{\Phi} F $ (see \autoref{def:PhiDim}) by replacing the size condition $ r^{1/\theta} \leq \abs{U_{i}} \leq r $ in \autoref{def:thetaDimByCover} with $ \Phi(r) \leq \abs{U_{i}} \leq r $, where $ \Phi $ is an admissible function. Here a function $ \Phi $ is called \textit{admissible} if there exists some $ Y > 0 $ such that $ \Phi $ is monotonic on $ (0, Y) $, and satisfies $ 0 < \Phi(r) \leq r $ for $ 0 < r < Y $ and $ \lim_{r\to 0} \Phi(r) / r = 0 $. In particular, we get the \tThetaDims\ when $ \Phi(r) = r^{1/\theta}$ $(0 < \theta < 1) $ and the box-counting dimensions when $ \Phi(r) = - r/\log r $ (see \cite[Proposition 3.2]{Banaji2020}).

It is natural to ask whether there are some results analogous to \autoref{thm:main} for the \tPhiDims. Our answer is affirmative. Moreover, our strategy can be exploited to study the \tPhiDims\ in several settings, including the orthogonal projections in Euclidean spaces and the images of fractional Brownian motions. For the clarity of illustration, we separately state the settings where we study the \tPhiDims.

\begin{setting}
	\label{set:selfaffine} Let $ T_{1}, \ldots,  T_{m} $ be a fixed family of contracting $ d\times d $ invertible real matrices. Write $ \Sigma = \{1,\ldots,m\}^{\bbN} $. For $ \bfa = (a_{1}, \ldots, a_{m}) \in \euclid[dm] $, let $ \pia\colon \Sigma \to \euclid $ be the coding map associated with the affine IFS $ \{ T_{j}x+a_{j}\}_{j=1}^{m} $ (see \eqref{eq:def-coding}).
\end{setting}

\begin{setting}
	\label{set:orthogon} Let $ G(d, m) $ be the Grassmannian of $ m $-dimensional subspaces of $ \euclid[d] $ and $ \gamma_{d,m} $ be the natural invariant probability measure on $ G(d,m) $. For $ V \in G(d,m) $, let $ \piV $ be the orthogonal projection from $ \euclid $ onto $ V $.
\end{setting}

\begin{setting}
	\label{set:brownian} For $ 0 < \alpha < 1 $, the \textit{index-$ \alpha $ fractional Brownian motion} is the Gaussian random function $ B_{\alpha} \colon \euclid \to \euclid[m] $ that with probability $ 1 $ is continuous with $ \Balph{0} = 0 $ and such that the increments $ \Balph{x} - \Balph{y} $ are multivariate normal with the mean vector $ 0 \in \euclid[m] $ and the covariance matrix $ \mathrm{diag}(\xyAbs^{2\alpha}, \ldots, \xyAbs^{2\alpha}) \in \euclid[m \times m]$. Denote the underlying probability space as $ (\Omega, \bbP) $. In particular, $ B_{\alpha} = (B_{\alpha, 1}, \ldots, B_{\alpha,m}) $, where $ B_{\alpha, i} \colon \euclid \to \bbR $ are independent index-$ \alpha $ fractional Brownian motions with distributions given by
	\begin{equation}\label{eq:FBM-dist}
		\bbP \{ B_{\alpha, i}(x) - B_{\alpha, i}(y) \in A \} = \frac{1}{\sqrt{2\pi} \xyAbs^{\alpha}} \int_{t\in A} \exp \left(- \frac{t^{2}}{2\xyAbs^{2\alpha} }\right) \, dt
	\end{equation}
	for each Borel set $ A \subset \bbR $.
\end{setting}

Now we are ready to present our results for the \tPhiDims\ using the \textit{generalized capacity dimensions} $ \lDim{C,\Phi} E$, $\uDim{C,\Phi} E$ (see \autoref{def:CapPhiDim}) and \textit{generalized dimension profiles} $ \lDim{\Phi}^{\tau} E$, $\lDim{\Phi}^{\tau} E $ (see \autoref{def:DimProfile}). 
\begin{theorem}\label{thm:PhiDim} Let $ \Phi $ be an admissible function.
	Suppose
	\begin{equation}\label{eq:GenThm-Cond}
		\lim_{r \to 0} r^{\varepsilon} \log \Phi(r) = 0 \mFor \text{all } \varepsilon > 0.
	\end{equation}
 	Then the followings hold.
	\begin{enumerate}[(i)]
		\item\label{itm:MT-affine} In \autoref{set:selfaffine}, let $ E \subset \Sigma $.
		     Then for all $ \bfa \in \euclid[dm] $,
		      \begin{equation*}
			      \lDim{\Phi} \pia(E) \leq \lDim{C,\Phi} E \quad\text{  and }\quad  \uDim{\Phi} \pia (E) \leq \uDim{C,\Phi} E.
		      \end{equation*}
		      Assume $ \norm{T_{j}} < 1/2 $ for $ 1 \leq j \leq m$. Then for $ \calL^{dm} $-a.e.\ $ \bfa \in \euclid[dm] $,
		      \begin{equation*}
			      \lDim{\Phi} \pia (E) = \lDim{C,\Phi} E \quad\text{  and }\quad  \uDim{\Phi} \pia (E) = \uDim{C,\Phi} E.
		      \end{equation*}

		\item\label{itm:MT-ortho} In \autoref{set:orthogon}, let $ E \subset \euclid $. Then for all $ V \subset G(d, m) $,
		      \begin{equation}\label{eq:GenThm-ortho-UB}
			      \lDim{\Phi} \piV E \leq \lDim{\Phi}^{m} E \quad\text{  and }\quad  \uDim{\Phi} \piV E \leq \uDim{\Phi}^{m} E.
		      \end{equation}
		      Moreover, for $ \gamma_{d,m} $-a.e.\ $ V\in G(d,m) $,
		      \begin{equation*}
			      \lDim{\Phi} \piV E = \lDim{\Phi}^{m} E \quad\text{  and }\quad  \uDim{\Phi} \piV E = \uDim{\Phi}^{m} E.
		      \end{equation*}

		\item\label{itm:MT-brown} In \autoref{set:brownian}, let $ E \subset \euclid[d] $. Then almost surely,
		      \begin{equation}\label{eq:GenThm-brown}
			      \lDim{\Phi} B_{\alpha}(E) = \frac{1}{\alpha}\lDim{\Phi_{\alpha}}^{\alpha m} E \quad\text{  and }\quad  \uDim{\Phi} B_{\alpha}(E) = \frac{1}{\alpha}\uDim{\Phi_{\alpha}}^{\alpha m} E,
		      \end{equation}
		      where $ \Phi_{\alpha} $ is defined in \eqref{eq:def-PhiAlpha}.
	\end{enumerate}
\end{theorem}
As is noted by Banaji, a question not explored in \cite{Banaji2020} is that whether the potential\nobreakdash-theoretic methods in \cite{BurrellEtAl2021,Burrell2021} can be adapted to study the \tPhiDims. This question is answered affirmatively by \autoref{thm:PhiDim} based on the kernels in \autoref{def:PhiKernels} and the condition \eqref{eq:GenThm-Cond}. Note that $ \eqref{eq:GenThm-Cond} $ holds if $ \liminf_{r\to 0} \log r /\log \Phi(r) > 0 $, which is satisfied by $ \Phi(r) = r^{1/\theta} $ ($ 0 < \theta < 1 $) and $ \Phi(r) =  - r/ \log r $. There are more general functions satisfying \eqref{eq:GenThm-Cond}, for example, $ \Phi(r) = r^{-\log r} $.

The paper is organizied as follows. In \autoref{sec:prelim}, we provide the definitions of the intermediate and capacity dimensions. Then the proofs of \ref{itm:UB} and \ref{itm:LB} of \autoref{thm:main} are respectively given in \autoref{sec:UB} and \autoref{sec:LB}. After introducing the generalized capacity dimensions and the generalized dimension profiles, we prove \autoref{thm:PhiDim} in \autoref{sec:PhiDim}. Finally, a few remarks are given in the last section.

\section{Preliminaries}\label{sec:prelim}

Throughout this paper, we shall mean by $ a \lesssim_{\varepsilon} b $ that $ a \leq Cb $ for some positive constant $ C $ depending on $ \varepsilon $. We write $ a \approx_{\varepsilon} b $ if $ a \lesssim_{\varepsilon} b $ and $ b \lesssim_{\varepsilon} a $. If it is clear from the context what $ C $ should depend on, we may briefly write by $ a \lesssim b $ or $ a\approx b $. We denote the natural logarithm by $ \log $  and the natural exponential by $ \exp $. By $ \# $ we denote the cardinality of a finite set. In a metric space, the closed ball centered at $x$ with radius $r$ is denoted by $B(x,r)$, and the closure of a set $E $ is denoted by $\overline{E}$.

\subsection{Intermediate dimensions} \label{sec:InterDims}
As noted in \cite{BurrellEtAl2021}, it is convenient to work with some equivalent definitions of the \tThetaDims. These definitions are expressed as limits of logarithms of sums over covers.
For $ s \geq 0 $, $ 0 < \theta \leq 1 $, and $ E \subset  \euclid $, define
\begin{equation*}
	\sumTheta{r}{E} \coloneqq  \inf \left \{\sum_{i} \abs{U_{i}}^{s} \colon \{U_{i}\} \text{ is a cover of } E \text{ with } r^{1/\theta}\leq \abs{U_{i}} \leq r \right  \}.
\end{equation*}

\begin{lemma}\label{def:thetaDim}
	Let $ 0 < \theta \leq 1 $. Then for $ E \subset \euclid $,
	\begin{equation*}
		\lDim{\theta} E = \left ( \text{the unique } s \in [0, d] \text{ such that } \liminf_{r\to 0} \frac{\log \sumTheta{r}{E}}{-\log r} = 0 \right)
	\end{equation*}
	and
	\begin{equation*}
		\uDim{\theta} E = \left ( \text{the unique } s \in [0, d] \text{ such that } \limsup_{r\to 0} \frac{\log \sumTheta{r}{E}}{-\log r} = 0 \right).
	\end{equation*}
\end{lemma}

\autoref{def:thetaDim} is a direct consequence of the following result.

\begin{lemma}[{\cite[Lemma 2.1]{BurrellEtAl2021}}]\label{lem:theta-CoverSumLimit-Cts}
	Let $ 0 < \theta \leq 1 $ and $ E \subset \euclid $. For $ 0 < r \leq 1 $ and $ 0\leq t\leq s \leq d $,
	\begin{equation*}
		- (1/\theta) (s-t) \leq \frac{\log \sumTheta{r}{E}}{-\log r} - \frac{\log \sumTheta[t]{r}{E} }{-\log r} \leq  -(s-t).
	\end{equation*}
	In particular, there is a unique $ s \in [0, d] $ such that $ \liminf_{r\to 0} \frac{\log \sumTheta{r}{E} }{-\log r} = 0 $ and  a unique $ s \in [0, d] $ such that $ \limsup_{r\to 0} \frac{\log \sumTheta{r}{E} }{-\log r} = 0 $.
\end{lemma}

\subsection{Capacity dimensions} \label{sec:CapDims}
Let $ T_{1}, \ldots,  T_{m} $ be a fixed family of contracting $ d\times d $ invertible real matrices. Recall $ \Sigma = \{1, \ldots, m\}^{\bbN} $. For $ n \in \bbN $, write $ \Sigma_{n} := \{1, \ldots, m\}^{n} $ and $ \Sigma^{\ast} := \Union_{n=1}^{\infty} \Sigma_{n} \Union \{\varnothing\}$, where $ \varnothing $ denotes the empty word. Write $ \abs{I} := n $ for $ I \in \Sigma_{n} $. For $ x = (x_{k})_{k=1}^{\infty} \in \Sigma $ and $ n \in \bbN $, denote $ x|n := x_{1}\cdots x_{n} $.
For $ I = i_{1} \cdots i_{n} \in \Sigma_{n} $, define
\begin{equation*}
	[I] := \{ x\in \Sigma \colon x|n = I \} \mAnd T_{I} := T_{i_{1}}\cdots T_{i_{n}}.
\end{equation*}
By convention we let $ x|0 = \varnothing $ and $ T_{\varnothing} $ be the identity map on $ \euclid $. Let $ \xWy $ denote the common initial segment of $ x, y \in \Sigma $. Endow $ \Sigma $ with the canonical metric $ d(x,y) := \exp(-\abs{\xWy}) $ for $ x, y \in \Sigma $. By convention we set $ \mu(I) := \mu([I]) $ for $ I \in \Sigma^{\ast} $ and any Borel measure $ \mu $ on $ \Sigma $. For any $ d\times d$ real matrix $ T $, the singular values of $ T $ are decreasingly denoted by $ \alpha_{1}(T), \ldots, \alpha_{d}(T) $.

Following \cite{JordanEtAl2007}, we define for $ r > 0 $, 
\begin{equation}\label{eq:def-kerZ}
	\kerZ{r}{\xWy} := \begin{cases}
		\prod_{k=1}^{d} \min \{1, \frac{r}{\alpha_{k}(T_{x\wedge y})}\} & \text{if }  x\neq y \\
		1                                                               & \text{if } x = y
	\end{cases} \mFor x,y\in \Sigma.
\end{equation}
For $ 0 \leq s \leq d  $, $ 0 <  \theta \leq 1 $, and $ r > 0 $, we introduce the \textit{kernel}
\begin{equation}\label{eq:def-kerTheta}
	\kerTheta{r}{\xWy} := \max_{r^{1/\theta} \leq u \leq r} u^{-s}\kerZ{u}{\xWy}  \mFor x, y \in \Sigma.
\end{equation}
Let $ \calP(E) $ denote the set of Borel probability measures supported on a compact set $ E $.
For compact set $ E \subset \Sigma $, the \textit{capacity} of $ E $ is defined by
\begin{equation}\label{eq:def-thetaCap}
	\capTheta{r}{E} = \left ( \inf_{\mu\in\calP(E)} \iint \kerTheta{r}{\xWy} \,d\mu(x)d\mu(y) \right )^{-1}.
\end{equation}
By convention we set $ \capTheta{r}{E} := \capTheta{r}{\ol{E}} $ for non-compact subset $ E \subset \Sigma $. Thus we can assume, without loss of generality, that the set whose capacities are considered is compact.

The existence of equilibrium measures for kernels and the relationship between the minimal energy and the corresponding potentials is standard in classical potential theory. We state this in a convenient form for the positive symmetric continuous kernels (cf.\ \cite[Theorem 2.4]{Fuglede1960} or \cite[Lemma 2.1]{Falconer2021}).

\begin{lemma}\label{lem:equil-meas}
	Let $ E $ be a non-empty compact set in a metric space, and let $ K \colon E \times E \to (0,+\infty) $ be a continuous function such that $ K(x,y) = K(y,x) $. Then there is some measure $ \mu_{0} \in \calP(E) $ such that
	\begin{equation*}
		\iint K(x,y) \, d\mu_{0}(x)d\mu_{0}(y) = \frac{1}{C_{K}(E)} ,
	\end{equation*}
	where $ C_{K}(E) = \left ( \inf_{\mu\in\calP(E)} \iint K(x,y)\,d\mu(x)d\mu(y) \right )^{-1} $. Moreover,
	\begin{equation}\label{eq:equil-meas}
		\int K(x,y) \, d\mu_{0}(y) \geq \frac{1}{C_{K}(E)} \quad \text{ for all } x\in E,
	\end{equation}
	with equality for $ \mu_{0} $-a.e.~$ x\in E $.
\end{lemma}
A measure $ \mu_{0} $ in \autoref{lem:equil-meas} is called an \textit{equilibrium measure} for the kernel $ K $. Now we are ready to introduce the capacity dimensions.
\begin{definition}[capacity dimensions]\label{def:theta-CapDim}
	Let  $ 0 < \theta \leq 1 $ and $ E \subset \Sigma $. The \textit{lower and upper capacity dimensions} of $ E $ are respectively defined by
	\begin{equation*}
		\lDim{C,\theta} E := \left ( \text{the unique } s \in [0, d] \text{ such that } \liminf_{r\to 0} \frac{\log \capTheta{r}{E} }{-\log r} = 0 \right ),
	\end{equation*}
	and
	\begin{equation*}
		\uDim{C,\theta} E := \left ( \text{the unique } s \in [0, d] \text{ such that } \limsup_{r\to 0} \frac{\log \capTheta{r}{E} }{-\log r} = 0 \right ).
	\end{equation*}
\end{definition}

\autoref{def:theta-CapDim} is justified by the following lemma, an analog of \cite[Lemma 3.2]{BurrellEtAl2021}. For completeness, we include a detailed proof.

\begin{lemma}\label{lem:theta-CapLimit-Cts}
	Let $ 0 < \theta \leq 1 $ and $ E \subset \Sigma $. Then for $ 0 < r \leq 1 $ and $ 0 \leq t \leq s \leq d $,
	\begin{equation}\label{eq:cts-LogCap}
		-(1/\theta) (s-t) \leq \frac{\log \capTheta[s]{r}{E} }{-\log r} -  \frac{\log \capTheta[t]{r}{E}}{-\log r} \leq -(s-t).
	\end{equation}
	In particular, there is a unique $ s \in [0, d] $ such that $ \liminf_{r\to 0} \frac{\log \capTheta{r}{E} }{-\log r} = 0 $ and a unique $ s \in [0, d] $ such that $ \limsup_{r\to 0} \frac{\log \capTheta{r}{E} }{-\log r} = 0 $.
\end{lemma}

\begin{proof}
	Since $ 0 < r \leq 1 $ and $ 0 \leq t \leq s $, it follows from \eqref{eq:def-kerTheta} that for $ x,y \in \Sigma $,
	\begin{equation}\label{eq:Jker-Compare}
		\begin{aligned}
		r^{(1/\theta) (s-t)} \kerTheta{r}{\xWy} & =  r^{(1/\theta)(s-t)} \max_{r^{1/\theta} \leq u \leq r}  u^{-s}\kerZ{u}{\xWy}                                           \\
		                                     & = r^{(1/\theta)(s-t)} \max_{r^{1/\theta} \leq u \leq r} u^{-(s-t)} u^{-t}\kerZ{u}{\xWy}                                  \\
		                                     & \leq \max_{r^{1/\theta} \leq u \leq r}  u^{-t}\kerZ{u}{\xWy}                          & \text{by } u \geq r^{1/\theta} \\
		                                     & = \kerTheta[t]{r}{\xWy} \\
		                                     & = \max_{r^{1/\theta} \leq u \leq r}  u^{s-t} u^{-s} \kerZ{u}{\xWy} \\ & \leq  r^{s-t} \kerTheta{r}{\xWy} & \text{by } u \leq r .
	\end{aligned}
	\end{equation}
	Without loss of generality assume that $ E $ is compact. By \autoref{lem:equil-meas}, there exists an equilibrium measure $ \mu_{0} $ on $ E $ for the kernel $ \kerTheta[t]{r}{\xWy} $. Then
	\begin{equation*}
		\begin{aligned}
		r^{(1/\theta)(s-t)} (\capTheta{r}{E})^{-1} & \leq r^{(1/\theta)(s-t)} \iint \kerTheta[s]{r}{\xWy} \,d\mu_{0}(x)d\mu_{0}(y) \\ & \leq \iint \kerTheta[t]{r}{\xWy} \,d\mu_{0}(x)d\mu_{0}(y) & \text{by } \eqref{eq:Jker-Compare}
		\\ & = (\capTheta[t]{r}{E})^{-1},
	\end{aligned}
	\end{equation*}
	and so
	\begin{equation}\label{eq:capTheta-t<s}
		r^{(1/\theta)(s-t)} \capTheta[t]{r}{E} \leq \capTheta[s]{r}{E}.
	\end{equation}
	Similarly,
	\begin{equation}\label{eq:capTheta-s<t}
		\capTheta[s]{r}{E} \leq  r^{s-t} \capTheta[t]{r}{E}.
	\end{equation}
	By \eqref{eq:capTheta-t<s} and \eqref{eq:capTheta-s<t}, taking logarithms and making a rearrangement give \eqref{eq:cts-LogCap}.

	Taking limits of the quotients in \eqref{eq:cts-LogCap} shows that the functions
	\begin{equation}\label{eq:capTheta-fcts}
		s \mapsto \liminf_{r\to 0} \frac{\log \capTheta{r}{E} }{-\log r} \mAnd s \mapsto  \limsup_{r\to 0} \frac{\log \capTheta{r}{E} }{-\log r}
	\end{equation}
	are strictly decreasing and continuous on $ [0, d] $. Since $ \kerTheta[0]{r}{\xWy} = \kerZ{r}{\xWy} \leq 1 $, we have $ \capTheta[0]{r}{E} \geq 1 $. This implies
	\begin{equation}\label{eq:capTheta-0}
		\liminf_{r\to 0} \frac{\log \capTheta[0]{r}{E} }{-\log r} \geq 0 \mAnd \limsup_{r\to 0} \frac{\log \capTheta[0]{r}{E} }{-\log r} \geq 0.
	\end{equation}
	On the other hand, since $ r^{-d} \kerZ{r}{\xWy} = \prod_{k=1}^{d} \min\{1/r, 1/\alpha_{k}(T_{\xWy})\} \geq 1 $ for $ 0 < r \leq 1 $, we have
	\begin{align*}
		\kerTheta[d]{r}{\xWy} = \max_{r^{1/\theta}\leq u\leq r } u^{-d} \kerZ{u}{\xWy} \geq 1.
	\end{align*}
	Hence $ \capTheta[d]{r}{E} \leq 1 $, and so
	\begin{equation}\label{eq:capTheta-d}
		\liminf_{r\to 0} \frac{\log \capTheta[d]{r}{E} }{-\log r} \leq 0 \mAnd \limsup_{r\to 0} \frac{\log \capTheta[d]{r}{E} }{-\log r} \leq 0.
	\end{equation}
	Based on \eqref{eq:capTheta-0} and \eqref{eq:capTheta-d}, the proof is completed by the continuity and strict monotonicity of the functions in \eqref{eq:capTheta-fcts}.
\end{proof}

\section{Proof of \autoref{thm:main}\ref{itm:UB}} \label{sec:UB}

We begin with a simple geometric observation.
\begin{lemma}[{\cite[Lemma 3.2]{FengEtAl2022}}]\label{lem:kerZ-counting}
	Let $ \mathbf{a} \in \euclid[dm] $. Then
	\begin{equation*}
		N_{r}\left( \pia([I]) \right) \lesssim_{d,\mathbf{a}} \frac{1}{\kerZ{r}{I}} \mFor  I \in \Sigma^{\ast} \text{ and } r > 0,
	\end{equation*}
	where $ N_{r}(F) $ denotes the minimal number of sets with diameter $ r $ needed to cover any bounded set $ F \subset \euclid $.
\end{lemma}

Next we deduce an upper bound on $ \sumTheta{r}{\pia(E)} $ from a lower bound on the potentials of a measure with respect to the kernel $ \kerTheta{r}{\xWy} $.

\begin{proposition}\label{lem:thetaDim-CoverSumUB}
	Let $ 0 \leq s \leq  d $, $ 0 < \theta \leq 1 $, and $  \mathbf{a} \in \euclid[dm] $. Let $ E \subset \Sigma $ be a non-empty compact set.  If for $ 0 < r \leq 1 $ there exist $ \mu \in \calP(E) $ and $ \gamma > 0 $ such that
	\begin{equation*}
		\int \kerTheta{r}{\xWy} \, d\mu(y) \geq \gamma \mFor \text{all } x \in E,
	\end{equation*}
	then for all sufficiently small $ r > 0 $,
	\begin{equation*}
	\sumTheta{r}{\pia(E)} \lesssim_{d,\bfa,\theta}   \frac{\log(1/r)}{\gamma}.
	\end{equation*}
\end{proposition}

We adapt some ideas from the proof of \cite[Lemma 4.4]{BurrellEtAl2021}. However, instead of only chopping the balls with too large diameters, we cover each of the balls having relatively large measure with sets of an appropriate diameter provided by the kernel $ \kerTheta{r}{\xWy} $.

\begin{proof}
	
	Let $ x \in E $. Set $ \ell(x) \coloneqq \min \{ n \in \bbN \colon \alpha_{1}(T_{x|n}) \leq r^{1/\theta} \} $. Write $ \alpha_{+} := \max_{1\leq j \leq m} \norm{T_{j}} $ and $ \ell := \lceil (1/\theta) \log r / \log \alpha_{+} \rceil $. Then $ \ell(x) \leq \ell $ since
	\begin{equation*}
		\alpha_{1}(T_{x|\ell}) = \norm{T_{x|\ell}} \leq \norm{T_{x_{1}}}\cdots\norm{T_{x_{\ell}}} \leq \alpha_{+}^{\ell} \leq r^{1/\theta}.
	\end{equation*}
	For $ n \geq \ell(x) $ and $ u \geq r^{1/\theta} $, it follows from \eqref{eq:def-kerZ} that $ \kerZ{u}{x|n} = 1 $, and so
	\begin{equation}\label{eq:thetaDim-kernel-const}
		\kerTheta{r}{x|n} = \max_{r^{1/\theta}\leq u \leq r } u^{-s}  = r^{-s/\theta}.
	\end{equation}
	This implies
	\begin{align*}
		\gamma & \leq \int \kerTheta{r}{\xWy} \, d\mu(y)                                                                                                                                         \\
		       & = \sum_{n=0}^{\ell(x)-1} \kerTheta{r}{x|n} \mu\{y\in\Sigma\colon\abs{\xWy} = n \}                                                                                           \\
		       & \qquad\qquad\qquad + r^{-s/\theta} \left ( \sum_{n=\ell(x)}^{\infty}  \mu\{y\in\Sigma\colon\abs{\xWy} = n \} + \mu(\{x\}) \right ) & \text{by } \eqref{eq:thetaDim-kernel-const} \\ & = \sum_{n=0}^{\ell(x)-1} \kerTheta{r}{x|n} \left [\mu(x|n)-\mu(x|(n+1))\right] + r^{-s/\theta} \mu(x|\ell(x)) \\
		       & \leq \sum_{n=0}^{\ell(x)} \kerTheta{r}{x|n} \mu(x|n)                                                                              & \text{by }\eqref{eq:thetaDim-kernel-const}.
	\end{align*}
	Hence there exists an integer $ n(x) \in [0, \ell(x)] $ such that
	\begin{equation}\label{eq:tmp121}
		\kerTheta{r}{x|n(x)} \mu(x|n(x)) \geq \frac{\gamma}{\ell(x) + 1 } \geq \frac{\gamma}{\ell + 1 }.
	\end{equation}
	By \eqref{eq:def-kerTheta}, there exists some $ \delta(x) \in [r^{1/\theta}, r] $ such that
	$ \kerTheta{r}{x|n(x)} =  \delta(x)^{-s} \kerZ{\delta(x)}{x|n(x)} $. Then \eqref{eq:tmp121} implies
	\begin{equation}\label{eq:tmp4}
		\delta(x)^{-s} \kerZ{\delta(x)}{x|n(x)} \mu(x|n(x)) \geq \frac{\gamma}{\ell + 1 }.
	\end{equation}

	Since $ \{ [x|n(x)]\}_{x\in E} $ is a cover of $ E $ and $ n(x) \leq \ell $, we can find a disjoint subcover $ \Gamma $ by the net structure of $ \Sigma $. By \eqref{eq:tmp4}, for each $ I \in \Gamma $ there exists some $ \delta_{I} \in [r^{1/\theta}, r] $ such that
	\begin{equation}\label{eq:counting<measure}
		\frac{\delta_{I}^{s}}{\kerZ{\delta_{I}}{I}} \leq \frac{\ell + 1}{\gamma} \mu(I).
	\end{equation}
	 Clearly $ \{ \pia([I]) \}_{I\in\Gamma} $ is a cover of $ \pia (E) $ since $ \Gamma $ covers $ E $. By \autoref{lem:kerZ-counting}, we can find for each $ \pia([I]) $ a cover $ \calD_{I} $ consisting of sets with diameter $ \delta_{I} \in [r^{1/\theta}, r] $ such that
	\begin{equation}\label{eq:chopping-number}
		\# \calD_{I} \lesssim_{d,\bfa} \frac{1}{\kerZ{\delta_{I}}{I}}.
	\end{equation}
	Then $ \Union_{I\in\Gamma} \calD_{I} $ is a cover of $ \pia(E) $ with sets of diameters in $  [r^{1/\theta}, r] $. Finally,
	\begin{align*}
	\sumTheta{r}{\pia(E)} \leq \sum_{I\in\Gamma} \sum_{B\in \calD_{I}} \abs{B}^{s}  & = \sum_{I\in\Gamma} 	\# \calD_{I} \cdot \delta_{I}^{s} \\ & \lesssim_{d,\bfa} \sum_{I\in\Gamma} 	\frac{\delta_{I}^{s}}{\kerZ{\delta_{I}}{I}} & \text{by } \eqref{eq:chopping-number} \\ & \leq  \sum_{I\in\Gamma} \frac{\ell + 1}{\gamma} \mu(I) & \text{by } \eqref{eq:counting<measure}   \\ & = \frac{\ell+1}{\gamma}.
	\end{align*}
	This finishes the proof since $ \ell + 1 \lesssim_{\theta} \log(1/r )$ when $ r $ is small.
\end{proof}

Now we are ready to prove \ref{itm:UB} of \autoref{thm:main}.

\begin{proof}[Proof of \autoref{thm:main}\ref{itm:UB}]
	Since the \tThetaDims\ and capacity dimensions of a set remain the same after taking closure, without loss of generality we can assume that $ E $ is compact.

	Let $ 0 \leq s \leq d $. For $ 0 < r \leq 1 $, by \autoref{lem:equil-meas} there exists an equilibrium measure $ \mu_{r} \in \calP(E) $ for the kernel $ \kerTheta{r}{\xWy} $ such that
	\begin{equation*}
		\int \kerTheta{r}{\xWy} \, d\mu_{r}(y) \geq \frac{1}{\capTheta{r}{E}} \quad \text{ for all } x\in E.
	\end{equation*}
	Applying \autoref{lem:thetaDim-CoverSumUB} with $ \mu = \mu_{r} $ gives that for all sufficiently small $ r > 0 $,
	\begin{equation*}
		\sumTheta{r}{\pia(E)} \lesssim \log(1/r) \cdot \capTheta{r}{E}.
	\end{equation*}
	By taking logarithms and limits,
	\begin{equation}\label{eq:nonInt-liminf}
		\liminf_{r\to 0} \frac{\log \sumTheta{r}{\pia(E)} }{-\log r} \leq \liminf_{r\to 0} \frac{\log \capTheta{r}{E} }{- \log r}
	\end{equation}
	and
	\begin{equation}\label{eq:nonInt-limsup}
		\limsup_{r\to 0} \frac{\log \sumTheta{r}{\pia (E)} }{- \log r} \leq \limsup_{r\to 0} \frac{\log \capTheta{r}{E} }{ - \log r}.
	\end{equation}
	Hence \autoref{def:thetaDim} and \autoref{def:theta-CapDim} show that
	\begin{equation*}
		\lDim{\theta} \pia (E) \leq \lDim{C,\theta} E \quad\text{  and }\quad  \uDim{\theta} \pia( E) \leq \uDim{C,\theta} E.
	\end{equation*}
	This completes the proof of \autoref{thm:main}\ref{itm:UB}.
\end{proof}

\section{Proof of \autoref{thm:main}\ref{itm:LB}} \label{sec:LB}

We begin with a lemma modified from \cite[Lemma 5.4]{BurrellEtAl2021}, which allows us to control $ \sumTheta{r}{\pia(E)} $ from below using the upper bounds on the potentials with respect to the kernel $ \kerPsiTheta{r}{s}{\xyAbs} $. For $ 0 \leq s \leq d $, $ 0 < \theta \leq 1 $ and $ 0 < r \leq 1 $, define
\begin{equation}\label{eq:def-truncatedKernel}
	\kerPsiTheta{r}{s}{\Delta} \coloneqq \begin{cases}
		r^{-s/\theta} & \text{if } 0 \leq \Delta \leq r^{1/\theta} \\
		\Delta^{-s}   & \text{if } r^{1/\theta} < \Delta \leq r    \\
		0             & \text{if }\Delta > r
	\end{cases}\quad \text{ for } \Delta \geq 0.
\end{equation}
Note that $ \kerPsiTheta{r}{s}{\cdot} $ is non-increasing.

\begin{lemma} \label{lem:thetaDim-CoverSumLB}
	Let $ 0 \leq s \leq d $, $ 0 < \theta \leq 1 $, and $ 0 < r \leq 1 $. Let $ E \subset \euclid[d]$ be a non-empty compact set. If there exist $ \mu \in \calP(E) $ and a Borel subset $ F \subset E $, and $ \gamma > 0 $ such that
	\begin{equation}\label{eq:thetaDim-potential-UB}
		\int \kerPsiTheta{r}{s}{\xyAbs} \, d\mu(y) \leq \gamma \mFor \text{all } x\in F,
	\end{equation}
	then
	\begin{equation*}
		\sumTheta{r}{E} \geq \frac{\mu(F)}{\gamma}.
	\end{equation*}
\end{lemma}

We can view \autoref{lem:thetaDim-CoverSumLB} as a potential-theoretic version of the mass distribution principle (see \cite{Falconer2003}). 


\begin{proof}
	 Let $ x \in F $ and $r^{1/\theta} \leq \delta \leq r $. It follows from \eqref{eq:def-truncatedKernel} that $ \kerPsiTheta{r}{s}{\xyAbs} \geq \delta^{-s}$ for $ y \in B(x, \delta) $. Then \eqref{eq:thetaDim-potential-UB} implies that
	\begin{equation*}
		\gamma \geq \int \kerPsiTheta{r}{s}{\xyAbs} \, d\mu(y) \geq \int_{B(x,\delta)}  \kerPsiTheta{r}{s}{\xyAbs} \, d\mu(y) \geq \frac{\mu(B(x,\delta))}{\delta^{s}},
	\end{equation*}
	thus
	\begin{equation}\label{eq:thetaDim-mdp}
		\delta^{s} \geq \frac{\mu(B(x,\delta))}{\gamma}.
	\end{equation}
	Let $ \{U_{i}\} $ be a cover of $ F $ with $ r^{1/\theta} \leq \abs{U_{i}} \leq r $. Without loss of generality we can pick some $ x_{i} \in F\intersect U_{i} $ for each $ i $. Then $ U_{i} \subset B(x_{i}, \abs{U_{i}}) $ for each $ i $. Hence by \eqref{eq:thetaDim-mdp},
	\begin{equation*}
		\sum_{i} \abs{U_{i}}^{s} \geq  \frac{1}{\gamma} \sum_{i} \mu(B(x_{i},\abs{U_{i}})) \geq \frac{\mu(F)}{\gamma}.
	\end{equation*}
	Taking infima over all such covers gives
	\begin{equation*}
		\sumTheta{r}{E} \geq 	\sumTheta{r}{F} \geq \frac{\mu(F)}{\gamma}.
	\end{equation*}
	This finishes the proof.
\end{proof}

The following lemma is contained in the proof of \cite[Lemma 5.1]{JordanEtAl2007} which verifies the so-called \textit{self-affine transversality} in \eqref{eq:affine-trans}. 


\begin{lemma}[{\cite[Lemma 5.1]{JordanEtAl2007}}] \label{lem:paramMeasure}
	Assume $ \norm{T_{j}} < 1/2 $ for $ 1 \leq j \leq m $. Let $ \rho > 0 $. Then for $ x, y \in \Sigma $ and $ r > 0 $,
	\begin{equation}\label{eq:affine-trans}
		\calL^{dm} \{ \mathbf{a}\in B_{\rho} \colon \DeltaPiXY \leq r \} \lesssim_{\rho,d} \kerZ{r}{\xWy},
	\end{equation}
	where $ B_{\rho} $ denotes the closed ball in $ \euclid[dm] $ centered at $ 0 $ with radius $ \rho $.
\end{lemma}

Next we exploit \autoref{lem:paramMeasure} to relate the integral of $ \bfa \mapsto \kerPsiTheta{r}{s}{\DeltaPiXY} $  to the kernel $ \kerTheta{r}{\xWy} $.

\begin{proposition}\label{lem:thetaDim-paraMeasure}
	Assume $ \norm{T_{j}} < 1/2 $ for $ 1 \leq j \leq m$. Let $ \rho > 0 $, $ 0 \leq s \leq d $ and $ 0 < \theta \leq 1 $. Then for $ x, y \in \Sigma $ and $ 0 < r \leq 1 $,
	\begin{equation}\label{eq:thetaDim-paramMeasure}
		\int_{B_{\rho}} \kerPsiTheta{r}{s}{\DeltaPiXY} \, d\mathbf{a} \lesssim_{\rho, d, \theta} \log(1/r) \kerTheta{r}{\xWy}.
	\end{equation}
\end{proposition}
\begin{proof}
	Write $ \calL := \calL^{dm}$ for short. By \autoref{lem:paramMeasure},
	\begin{equation}\label{eq:thetaDim-AppParaMeasure}
		\calL \{ \mathbf{a}\in B_{\rho} \colon \DeltaPiXY \leq r \} \lesssim_{\rho,d} \kerZ{r}{\xWy}.
	\end{equation}
	Set $ \Delta_{\mathbf{a}} := \DeltaPiXY $. Then
	\begin{align*}
		 \quad \; \int_{B_{\rho}} & \kerPsiTheta{r}{s}{\Delta_{\bfa}}  \, d\mathbf{a}                                                                                                                                                                                                                \\
		 & = \int_{\{\mathbf{a}\in B_{\rho}\colon \Delta_{\mathbf{a}} \leq r^{1/\theta} \}} r^{-s/\theta} \, d\mathbf{a} + \int_{\{\bfa \in B_{\rho}\colon r^{1/\theta} < \Delta_{\mathbf{a}} \leq r \}} \Delta_{\mathbf{a}}^{-s} \, d\mathbf{a}  \hspace{8em} \text{by \eqref{eq:def-truncatedKernel}} \\
		 & = r^{-s/\theta} \calL\{\mathbf{a}\in B_{\rho}\colon \Delta_{\mathbf{a}} \leq r^{1/\theta} \} +  \int_{0}^{\infty} \calL \{\mathbf{a}\in B_{\rho}\colon  r^{1/\theta}< \Delta_{\mathbf{a}} \leq r,\; \Delta_{\mathbf{a}}^{-s} \geq t \} \, dt                                              \\
		 & = r^{-s/\theta} \calL\{\mathbf{a}\in B_{\rho}\colon \Delta_{\mathbf{a}} \leq r^{1/\theta} \} + \int_{0}^{r^{-s/\theta}} \calL \{\mathbf{a}\in B_{\rho}\colon r^{1/\theta} <\Delta_{\mathbf{a}} \leq \min\{r, t^{-1/s}\} \} \, dt                                                          \\
		 & =   \int_{0}^{r^{-s/\theta}} \calL \{\mathbf{a}\in B_{\rho}\colon \Delta_{\mathbf{a}} \leq \min\{r, t^{-1/s}\} \} \, dt                                                                                                                                                                   \\
		 & =  \int_{0}^{r^{- s}} \calL \{\mathbf{a}\in B_{\rho}\colon  \Delta_{\mathbf{a}} \leq r \} \, dt +  \int_{r^{-s}}^{r^{-s/\theta}} \calL \{\mathbf{a}\in B_{\rho}\colon  \Delta_{\mathbf{a}} \leq t^{-1/s} \} \, dt                                                                         \\
		 & \lesssim_{d,\rho} r^{-s} \kerZ{r}{\xWy} +  \int_{r^{-s}}^{r^{-s/\theta}} \kerZ{t^{-1/s}}{\xWy} \, dt  \hspace{13em} \text{by \eqref{eq:thetaDim-AppParaMeasure} }                                                                                                                                  \\
		 & = r^{-s} \kerZ{r}{\xWy} + s \int_{r^{1/\theta}}^{r} u^{-(s+1)} \kerZ{u}{\xWy} \, du   \hspace{7em} \text{ by taking } u = t^{-1/s}                                                                                                                                                       \\
		 & \leq  \left ( 1 +  s \int_{r^{1/\theta}}^{r} u^{-1}  \, du    \right )  \kerTheta{r}{\xWy}         \hspace{16.6em} \text{by } \eqref{eq:def-kerTheta} \\
		 & \lesssim_{d,\theta} \log(1/r) \kerTheta{r}{\xWy},
	\end{align*}
	where the last inequality is by $ s \int_{r^{1/\theta}}^{r} u^{-1}  \, du = s(1/\theta - 1)\log(1/r) \lesssim_{d,\theta} \log(1/r) $.
\end{proof}

Now we are ready to prove \autoref{thm:main}\ref{itm:LB}.

\begin{proof}[Proof of \autoref{thm:main}\ref{itm:LB}] Our arguments are mainly adapted from the proof of \cite[Theorem 5.1]{BurrellEtAl2021}. We focus on the case of the upper \tThetaDims\ while the proof for the lower \tThetaDims\ is similar. By \autoref{thm:main}\ref{itm:UB}, it suffices to prove
	\begin{equation*}
		\uDim{\theta} \pia(E) \geq \uDim{C,\theta} E
	\end{equation*}
	for $ \calL^{dm} $-a.e.\ $ \bfa\in B_{\rho} $ and $ \rho > 0 $.

	Let $ 0 \leq s \leq d $. Take a sequence $ (r_{k})_{k=1}^{\infty} $ tending to $ 0 $ such that $ 0 < r_{k} \leq 2^{-k} $ and
	\begin{equation}\label{eq:thetaDim-rk}
		\limsup_{k\to\infty} \frac{\log \capTheta{r_{k}}{E} }{-\log r_{k}} = \limsup_{r\to 0} \frac{\log  \capTheta{r}{E} }{-\log r}.
	\end{equation}
	By \autoref{lem:equil-meas}, for each $ k \in \bbN $ there is an equilibrium measure $ \mu_{k} $  on $ E $ for the kernel $ \kerTheta{r_{k}}{\xWy} $. Write
	\begin{equation*}
		\gamma_{k} \coloneqq \frac{1}{ \capTheta{r_{k}}{E} } = \iint \kerTheta{r_{k}}{\xWy} \, d\mu_{k}(x) d\mu_{k}(y).
	\end{equation*}
	Let $ \rho > 0 $. \autoref{lem:thetaDim-paraMeasure} implies that
	\begin{equation}\label{eq:thetaDim-UB-Integral}
		\iint \int_{B_{\rho}} \kerPsiTheta{r_{k}}{s}{\DeltaPiXY} \,d\mathbf{a} \; d\mu_{k}(x) d\mu_{k}(y) \lesssim \gamma_{k}\log(1/r_{k}).
	\end{equation}
	Let $ \varepsilon > 0 $. Note that there is some $ A > 0 $ such that $ r^{\varepsilon/2}\log (1/r) \leq  A $ for all $ r > 0 $. Then summing \eqref{eq:thetaDim-UB-Integral} over $ k\in\bbN $ and using Fubini's theorem lead to
	\begin{align*}
		\int_{B_{\rho}} & \sum_{k=1}^{\infty} \left (  r_{k}^{\varepsilon} \gamma_{k}^{-1} \iint  \kerPsiTheta{r_{k}}{s}{\DeltaPiXY} \, d\mu_{k}(x) d\mu_{k}(y) \right) d\mathbf{a}                \\
		                & \lesssim \sum_{k=1}^{\infty} \log(1/r_{k}) r_{k}^{\varepsilon} \leq A \sum_{k=1}^{\infty} r_{k}^{\varepsilon/2} \leq A \sum_{k=1}^{\infty} 2^{-k\varepsilon/2} < \infty.
	\end{align*}
	Hence for $ \calL^{dm} $-a.e.\ $ \mathbf{a} \in B_{\rho} $, there exists $ M_{\mathbf{a}} > 0 $ such that
	\begin{equation*}
		\iint  \kerPsiTheta{r_{k}}{s}{\abs{u-v}} \, d\pia\mu_{k}(v) d\pia\mu_{k}(u) \leq M_{\mathbf{a}}\gamma_{k} r_{k}^{-\varepsilon} \mFor \text{all } k \in \bbN.
	\end{equation*}
	Then for each $ k $ there exists some Borel $ F_{k} \subset\pia( E) $ such that $ (\pia\mu_{k})(F_{k}) \geq 1/2 $ and
	\begin{equation*}
		\int  \kerPsiTheta{r_{k}}{s}{\abs{u-v}} \, d\pia\mu_{k}(v) \leq 2M_{\mathbf{a}}\gamma_{k} r_{k}^{-\varepsilon} \mFor \text{all } u \in F_{k}.
	\end{equation*}
	\autoref{lem:thetaDim-CoverSumLB} implies that
	\begin{equation*}
	\sumTheta{r_{k}}{\pia(E)} \geq \frac{1}{2}  (2M_{\mathbf{a}}\gamma_{k} r_{k}^{-\varepsilon})^{-1} \gtrsim_{\bfa} r_{k}^{\varepsilon}\gamma_{k}^{-1} = r_{k}^{\varepsilon}\capTheta{r_{k}}{E},
	\end{equation*}
	thus
	\begin{align*}
		\limsup_{k\to\infty} \frac{\log \sumTheta{r_{k}}{\pia(E)} }{-\log r_{k}}
		                                                                             & \geq \limsup_{k\to\infty} \frac{\log \left( r_{k}^{\varepsilon}\capTheta{r_{k}}{E}\right)  }{-\log r_{k}} \\
		                                                                             & = -\varepsilon + \limsup_{k\to\infty} \frac{\log \capTheta{r_{k}}{E}  }{-\log r_{k}} \\ 
		                                                                             & = -\varepsilon + \limsup_{r\to 0} \frac{\log \capTheta{r}{E} }{-\log r} & \text{by } \eqref{eq:thetaDim-rk}.
	\end{align*}
	Letting $ \varepsilon \to 0 $ gives
	\begin{equation*}
		\limsup_{r\to 0}  \frac{\log \sumTheta{r}{\pia(E)} }{-\log r} \geq \limsup_{r\to 0} \frac{\log \capTheta{r}{E} }{-\log r} \mFor 0 \leq s \leq d.
	\end{equation*}
	Finally  \autoref{def:thetaDim} and \autoref{def:theta-CapDim} show that 
	\begin{equation*}
		\uDim{\theta} \pia(E) \geq \uDim{C,\theta} E \mFor \calL^{dm}\text{-a.e. } \mathbf{a} \in B_{\rho}.
	\end{equation*}
	The proof for the lower \tThetaDims\ is similar.
\end{proof}

\section{Generalized intermediate dimensions}\label{sec:PhiDim}
In this section, we will prove \autoref{thm:PhiDim} through a similar strategy of \autoref{thm:main}. In what follows, let $ \Phi \colon (0, Y) \to (0, \infty) $ be an admissible function for some $ Y > 0 $.

\subsection{Generalized intermediate dimensions} Following \cite{Banaji2020}, we introduce the generalized intermediate dimensions called the \textit{\tPhiDims}.

\begin{definition}[\tPhiDims]\label{def:PhiDim}
	For $ E \subset \euclid $, its \textit{upper \tPhiDim} is defined by
	\begin{align*}
		\uDim{\Phi} & E = \inf \{ s \geq 0 \colon \text{ for all } \varepsilon > 0 \text{ there exists } r_{0} \in (0, 1] \text{ such that for all } r \in (0, r_{0}),                                       \\
		            & \text{there exists a cover } \{U_{i} \} \text{ of } E \text{ such that }  \Phi(r) \leq \abs{U_{i}} \leq r \text{ for all } i \text{ and } \sum_{i} \abs{U_{i}}^{s} \leq \varepsilon \}
	\end{align*}
	and its \textit{lower \tPhiDim} is defined by
	\begin{align*}
		\lDim{\Phi} E = \inf \{ & s \geq 0\colon \text{ for all } \varepsilon > 0  \text{ and } r_{0} \in (0, 1] \text{ there exists } r \in (0, r_{0}) \text{  and }                                         \\
		                         & \text{ a cover } \{U_{i} \} \text{ of } E \text{ such that }  \Phi(r) \leq \abs{U_{i}} \leq r \text{ for all } i \text{ and } \sum_{i} \abs{U_{i}}^{s} \leq \varepsilon \}.
	\end{align*}
\end{definition}

We can describe the \tPhiDims\ by employing a similar approach to that used in defining the Hausdorff dimension with the aid of the Hausdorff measures. For $ s \geq 0 $, $ r > 0 $, and $ E \subset \euclid $, define
\begin{equation}\label{eq:def-SumPhi}
	\sumPhi{r}{E} = \inf \left \{\sum_{i} \abs{U_{i}}^{s} \colon \{U_{i}\} \text{ is a cover of } E \text{ with } \Phi(r) \leq \abs{U_{i}} \leq r \right  \}.
\end{equation}
\begin{lemma}\label{lem:char-PhiDim} Let $ \Phi $ be an admissible function and $ E \subset \euclid $. Then
	\begin{align*}
		\uDim{\Phi} E & = \inf \{s \geq 0 \colon \limsup_{r\to 0} \sumPhi{r}{E} < \infty \} = \inf \{s \geq 0 \colon \limsup_{r\to 0} \sumPhi{r}{E} = 0 \} \\
		              & = \sup \{s \geq 0 \colon \limsup_{r\to 0} \sumPhi{r}{E} = \infty \} = \sup \{s \geq 0 \colon \limsup_{r\to 0} \sumPhi{r}{E} > 0 \}
	\end{align*}
	and
	\begin{align*}
		\lDim{\Phi} E & = \inf \{s \geq 0 \colon \liminf_{r\to 0} \sumPhi{r}{E} < \infty \} = \inf \{s \geq 0 \colon \liminf_{r\to 0} \sumPhi{r}{E} = 0 \}  \\
		              & = \sup \{s \geq 0 \colon \liminf_{r\to 0} \sumPhi{r}{E} = \infty \} = \sup \{s \geq 0 \colon \liminf_{r\to 0} \sumPhi{r}{E} > 0 \}.
	\end{align*}
\end{lemma}
\begin{proof}	
	Since $ E $ is non-empty, we have $ \sumPhi[0]{r}{E} \geq 1 $. Pick any $ x \in E $, then $ E \subset B(x, \abs{E}) $. Since $ E $ is bounded, we have for $ r \leq \abs{E} $,
	\begin{align*}
		\sumPhi[d]{r}{E}  \leq \sumPhi[d]{r}{B(x, \abs{E})} \leq r^{d} N_{r}(B(x, \abs{E})) \lesssim_{d} r^{d}\max \{1, \frac{\abs{E}^{d}}{r^{d}} \} = \abs{E}^{d},
	\end{align*}
	where the last inequality follows from \autoref{lem:kerGeo-counting}.
	Note that for $ 0 \leq t\leq s $,
	\begin{equation}\label{eq:sumPhi-compare}
		\Phi(r)^{s-t} \sumPhi[t]{r}{E} \leq \sumPhi[s]{r}{E} \leq r^{s-t} \sumPhi[t]{r}{E}.
	\end{equation}
	By combining \eqref{eq:sumPhi-compare} with $ \sumPhi[0]{r}{E} \geq 1$ and $ \sumPhi[d]{r}{E} \lesssim \abs{E}^{d} $, we can complete the proof in a similar manner like the definition of the Hausdorff dimensions (see \cite[Section 3.2]{Falconer2003}).
\end{proof}

Note that by \eqref{eq:sumPhi-compare} a similar proof of \autoref{def:thetaDim} shows that 
\begin{equation}\label{eq:char-lPhiDim-zero}
	\uDim{\Phi} E = \left ( \text{the unique } s \in [0, d] \text{ such that } \limsup_{r\to 0} \frac{\log \sumPhi{r}{E}}{-\log r} = 0 \right)
\end{equation}
if $ \limsup_{r \to 0} \log r / \log \Phi(r) > 0 $, and
\begin{equation}\label{eq:char-uPhiDim-zero}
	\lDim{\Phi} E = \left ( \text{the unique } s \in [0, d] \text{ such that } \liminf_{r\to 0} \frac{\log \sumPhi{r}{E}}{-\log r } = 0 \right).
\end{equation}
if $ \liminf_{r \to 0} \log r / \log \Phi(r) > 0 $.

\subsection{Generalized capacity dimensions and dimension profiles} \label{subsec:def-CapDimProfiles}

We begin with the introduction of some appropriate kernels in the corresponding settings.

\begin{definition}(kernels) \label{def:PhiKernels} Let $ \Phi $ be an admissible function.
	\begin{itemize}
		\item In \autoref{set:selfaffine}, let $ 0 \leq s \leq d $ and $ 0 < r \leq 1 $. Define
		      \begin{equation}\label{eq:def-ker-affine-phi}
			      \kerPhi{r}{\xWy} := \max_{\Phi(r) \leq u \leq r} u^{-s} \kerZ{u}{\xWy} \mFor x , y\in \Sigma,
		      \end{equation}
	      where $ \kerZ{u}{\xWy} $ is defined in \eqref{eq:def-kerZ}.

		\item In \autoref{set:orthogon} and \autoref{set:brownian}, let $ \tau > 0 $, $ 0 \leq s \leq \tau $, and $ 0 < r \leq 1 $. Define
		      \begin{equation}\label{eq:def-ker-profile-phi}
			      \begin{aligned}
				      \kerProf{r}{s}{\tau}{\xyAbs} := \max_{\Phi(r) \leq u \leq r} u^{-s} \kerGeoCount{u}{\tau}{\xyAbs} \mFor  x, y \in \euclid[d],
			      \end{aligned}
		      \end{equation}
		      where
		      \begin{equation}\label{eq:def-ker-geoCounting}
			      \kerGeoCount{u}{\tau}{\Delta} := \min\left \{1, \left (\frac{u}{\Delta}\right )^{\tau}\right \} \quad \text{ for }  \Delta \geq 0.
		      \end{equation}
	\end{itemize}
\end{definition}

Like \autoref{lem:kerZ-counting}, we have the following simple geometric fact.

\begin{lemma}\label{lem:kerGeo-counting}
	Let $ B(x, \Delta) \subset \euclid[m] $ be a ball. Then for $ r > 0 $,
	\begin{equation*}
		N_{r}(B(x,\Delta)) \lesssim_{m} \frac{1}{\kerGeoCount{r}{m}{\Delta}}.
	\end{equation*}
\end{lemma}
\begin{proof}
 	Write $ x = (x_{1}, \ldots, x_{m}) $ and $ A = \prod_{i = 1}^{m} [x_{i}-\Delta, x_{i} + \Delta] $. Note that $ A $ can be covered by $ C \max \{ 1, (\Delta/r)^{m} \} $ manly cubes with side length $ r/\sqrt{m} $, where $ C $ is a constant only depending on $ m $. This completes the proof since $ B(x, \Delta) \subset A $ and the diameter of each cube with side length $ r/\sqrt{m} $ is $ r $.
\end{proof}

We proceed by defining the capacities with respect to the above kernels.

\begin{definition}[capacities]\label{def:gen-Cap}
	Let $ X $ be a compact metric space and  $ K \colon X \times X \to (0, +\infty) $ be a continuous function. For each compact set $ E \subset X $, the \textit{capacity} of $ E $ with respect to the \textit{kernel} $ K $ is defined by
	\begin{equation}\label{eq:def-PhiCap}
		C_{K}(E) := \left ( \inf_{\mu\in\calP(E)} \iint K(x,y) \, d\mu(x)d\mu(y) \right )^{-1}.
	\end{equation}
	By convention we set $ C_{K}(E) = C_{K}(\ol{E}) $ for every non-compact set $ E \subset X $. Thus when it comes to capacities, without loss of generality we can assume that the underlying set is compact. In particular, we focus on the following capacities.
	\begin{itemize}
		\item In \autoref{set:selfaffine}, let $  K(x,y) = \kerPhi{r}{\xWy} $ (see \eqref{eq:def-ker-affine-phi}). Define
		      \begin{equation*}
			      \capPhi{r}{E} := C_{K}(E) \mFor E \subset \Sigma.
		      \end{equation*}

		\item In \autoref{set:orthogon} and \autoref{set:brownian}, let $ K(x, y) = \kerProf{r}{s}{\tau}{\xyAbs} $ (see \eqref{eq:def-ker-profile-phi}). Define
		      \begin{equation*}
			      \capProf{r}{s}{\tau}{E} := C_{K}(E) \mFor \text{each bounded set }E \subset \euclid[d].
		      \end{equation*}
	\end{itemize}
\end{definition}

Now we are ready to define the generalized capacity dimensions called the \textit{$ \Phi $\nobreakdash-capacity dimensions} and the generalized dimension profiles called the \textit{$ \Phi $\nobreakdash-dimension profiles}.

\begin{definition}[$ \Phi $-capacity dimensions]\label{def:CapPhiDim} In \autoref{set:selfaffine}, let $ E \subset \Sigma $. The \textit{upper and lower $ \Phi $-capacity dimensions} of $ E $ are respectively defined by
		     \begin{align*}
			      \uDim{C,\Phi} E & = \inf \{s \geq 0 \colon \limsup_{r\to 0} \capPhi{r}{E} < \infty \} = \inf \{s \geq 0 \colon \limsup_{r\to 0} \capPhi{r}{E} = 0 \} \\
			                      & = \sup \{s \geq 0 \colon \limsup_{r\to 0} \capPhi{r}{E} = \infty \} = \sup \{s \geq 0 \colon \limsup_{r\to 0} \capPhi{r}{E} > 0 \}
		      \end{align*}
		      and
		      \begin{align*}
			      \lDim{C,\Phi} E & = \inf \{s \geq 0 \colon \liminf_{r\to 0} \capPhi{r}{E} < \infty \} = \inf \{s \geq 0 \colon \liminf_{r\to 0} \capPhi{r}{E} = 0 \}  \\
			                      & = \sup \{s \geq 0 \colon \liminf_{r\to 0} \capPhi{r}{E} = \infty \} = \sup \{s \geq 0 \colon \liminf_{r\to 0} \capPhi{r}{E} > 0 \}.
		      \end{align*}
\end{definition}

\begin{definition}[$ \Phi $\nobreakdash-dimension profiles]\label{def:DimProfile} In \autoref{set:orthogon} and \autoref{set:brownian}, let $ E \subset \euclid $ and $ \tau > 0 $. The \textit{upper and lower $ \Phi $-dimension profiles} of $ E $ are respectively defined by
	\begin{align*}
		\uDim{\Phi}^{\tau} E & = \inf \{s \geq 0 \colon \limsup_{r\to 0} \capProf{r}{s}{\tau}{E} < \infty \} = \inf \{s \geq 0 \colon \limsup_{r\to 0} \capProf{r}{s}{\tau}{E} = 0 \} \\
		& = \sup \{s \geq 0 \colon \limsup_{r\to 0} \capProf{r}{s}{\tau}{E} = \infty \} = \sup \{s \geq 0 \colon \limsup_{r\to 0} \capProf{r}{s}{\tau}{E} > 0 \}
	\end{align*}
	and
	\begin{align*}
		\lDim{\Phi}^{\tau} E & = \inf \{s \geq 0 \colon \liminf_{r\to 0} \capProf{r}{s}{\tau}{E} < \infty \} = \inf \{s \geq 0 \colon \liminf_{r\to 0} \capProf{r}{s}{\tau}{E} = 0 \}  \\
		& = \sup \{s \geq 0 \colon \liminf_{r\to 0} \capProf{r}{s}{\tau}{E} = \infty \} = \sup \{s \geq 0 \colon \liminf_{r\to 0} \capProf{r}{s}{\tau}{E} > 0 \}.
	\end{align*}
\end{definition}

\autoref{def:CapPhiDim} and \autoref{def:DimProfile} are justified as follows. Let $ K_{\Phi, r}^{s}(x,y) = \kerPhi{r}{\xWy} $ or $ \kerProf{r}{s}{\tau}{\xyAbs} $. According to \autoref{def:PhiKernels}, for $  0 \leq t \leq s $,
\begin{equation*}
	r^{-(s-t)} K_{\Phi, r}^{t}(x,y) \leq K_{\Phi, r}^{s}(x,y) \leq \Phi(r)^{-(s-t)} K_{\Phi, r}^{t}(x,y).
\end{equation*}
Then
\begin{equation}\label{eq:GenDim-justify}
	\Phi(r)^{s-t} C_{ K_{\Phi, r}^{t} }(E) \leq C_{ K_{\Phi, r}^{s} }(E) \leq r^{s-t} C_{ K_{\Phi, r}^{t} }(E).
\end{equation}
Since $ \kerPhi[0]{r}{\xWy} \leq 1 $ and $ \kerPhi[d]{r}{\xWy} \geq 1 $, we have $ \capPhi[0]{r}{E} \geq 1 $ and $ \capPhi[d]{r}{E} \leq 1 $. Hence for $ E \subset \Sigma $,
\begin{equation} \label{eq:CapDim-Bounds}
	0 \leq \lDim{C,\Phi} E \leq \uDim{C,\Phi} E \leq d. 
\end{equation}
Let $ E \subset \euclid $. Since $ \kerProf{r}{0}{\tau}{\xyAbs} \leq 1 $ and for $ 0 <  r < \min\{ 1, \abs{E}\} $,
\begin{equation*}
	\kerProf{r}{\tau}{\tau}{\xyAbs} = \max_{\Phi(r) \leq u \leq r} \min\left \{\frac{1}{u^{\tau}},  \frac{1}{\xyAbs^{\tau}}\right \}  \geq \frac{1}{\abs{E}^{\tau}} \gtrsim_{\tau} 1,
\end{equation*}
we have $ \capProf{r}{0}{\tau}{E} \geq 1 $ and $ \capProf{r}{\tau}{\tau}{E} \lesssim 1 $ when $ r $ is small. Hence
\begin{equation}\label{eq:DimProf-Bounds}
	0 \leq \lDim{\Phi}^{\tau} E \leq \uDim{\Phi}^{\tau} E \leq \tau.
\end{equation}
A combination of \eqref{eq:GenDim-justify}, \eqref{eq:CapDim-Bounds}, and \eqref{eq:DimProf-Bounds} justifies \autoref{def:CapPhiDim} and \autoref{def:DimProfile} in the same way as the proof of \autoref{lem:char-PhiDim}.

From \eqref{eq:GenDim-justify}, we can characterize the $ \Phi $\nobreakdash-capacity dimensions and the $ \Phi $\nobreakdash-dimension profiles like \eqref{eq:char-lPhiDim-zero} and \eqref{eq:char-uPhiDim-zero} if $ \limsup_{r\to0} \log r/ \log \Phi(r) >  0 $ and $ \liminf_{r\to0} \log r / \log \Phi(r) > 0 $ are respectively assumed.

Before finishing this subsection, we introduce a function $ \Phi_{\alpha} $ associated with $ \Phi $. It is useful in the proof of \autoref{thm:PhiDim}. For $ 0 < \alpha \leq 1 $, define
\begin{equation}\label{eq:def-PhiAlpha}
	\Phi_{\alpha} (r) := \Phi(r^{\alpha})^{1/\alpha} \mFor 0 < r < Y^{1/\alpha}.
\end{equation}
It is readily checked that $ \Phi_{\alpha} $ is admissible.

\subsection{Upper bound}
We begin with a lemma about the behavior of the capacities under the H\"older continuous maps.

\begin{lemma}\label{lem:HolderCap}
	 Let $ \tau > 0 $ and $ 0 \leq s \leq \tau $. Let $ f\colon \euclid \to \euclid[m] $ be a map. If there is some $ 0 < \alpha \leq 1 $ such that
	\begin{equation}\label{eq:holder}
		\abs{f(x) - f(y)} \lesssim \abs{x-y}^{\alpha} \mFor x,y \in \euclid,
	\end{equation}
	then for $ 0 < r \leq 1 $,
	\begin{equation}\label{eq:holder-ker}
		\kerProf{r}{s}{\tau}{\abs{f(x)-f(y)}} \gtrsim_{\tau}  \kerAlphaPhi{r^{1/\alpha}}{\alpha s}{\alpha \tau}{\abs{x-y}}.
	\end{equation}
	In particular, for $ E \subset \euclid $ and $ 0 < r \leq 1 $,
	\begin{equation*}
		\capProf{r}{s}{\tau}{f(E)} \lesssim_{\tau} \capAlphaPhi{r^{1/\alpha}}{\alpha s}{\alpha \tau}{E}.
	\end{equation*}
\end{lemma}
\begin{proof}
	According to \eqref{eq:def-ker-profile-phi},
	\begin{equation}\label{eq:ker-holder}
		\begin{aligned}
		& \!\!\!\!\!\!\!\! \kerProf{r}{s}{\tau}{\abs{f(x)-f(y)}} \\ & = \max_{\Phi(r) \leq u \leq r} u^{-s} \min \left \{ 1, \frac{u^{\tau}}{\abs{f(x)-f(y)}^{\tau}}\right \} \\
		& \gtrsim_{\tau} \max_{\Phi(r) \leq u \leq r} u^{- s} \min \left \{ 1, \frac{u^{\tau}}{\xyAbs^{\alpha \tau}}\right \} & \text{by }\eqref{eq:holder}        \\
		& = \max_{\Phi(r)^{1/\alpha} \leq v \leq r^{1/\alpha}} v^{-\alpha s} \min \left \{ 1, \frac{v^{\alpha \tau}}{\abs{x-y}^{\alpha \tau}}\right \}    & \text{by letting } v = u^{1/\alpha}        \\ 
		& = \max_{\Phi_{\alpha}(r^{1/\alpha}) \leq v \leq r^{1/\alpha}} v^{-\alpha s} \min \left \{ 1, \frac{v^{\alpha \tau}}{\abs{x-y}^{\alpha \tau}}\right \}      & \text{by } \eqref{eq:def-PhiAlpha} \\
		& = \kerProfAlpha{r}{s}{\tau}{\xyAbs} .
	\end{aligned}
	\end{equation}
	This proves \eqref{eq:holder-ker}.

	Without loss of generality we assume that $ E $ is compact. By \autoref{lem:equil-meas}, there is an equilibrium measure $ \nu \in \calP(f(E)) $ for the kernel $  \kerProf{r}{s}{m}{\abs{u-v}}  $. Then \cite[Theorem 1.20]{Mattila1995} gives some $ \mu\in\calP(E) $ such that $ \nu = f\mu $. Hence
	\begin{equation}\label{eq:Holder-Cap}
		\begin{aligned}
			\capProfAlpha{r}{s}{\tau}{E} & \geq \left (\int  \kerProfAlpha{r}{s}{\tau}{\xyAbs} \, d\mu(x) d\mu(y) \right )^{-1} \\ & \gtrsim_{\tau} 	\left (\int  \kerProf{r}{s}{\tau}{\abs{f(x)-f(y)}} \, d\mu(x) d\mu(y) \right )^{-1} & \text{by } \eqref{eq:holder-ker} \\ & = \left (\int  \kerProf{r}{s}{\tau}{\abs{u-v}} \, d\nu(u) d\nu(v) \right )^{-1} = \capProf{r}{s}{\tau}{f(E)}.
	\end{aligned}
	\end{equation}
	This completes the proof.
\end{proof}

We now demonstrate how the capacities behave under a map with the modulus of continuity similar to that of  the fractional Brownian motion.

\begin{lemma}\label{lem:FBM-Cap}
		 Let $ \tau > 0 $ and $ 0 \leq s \leq \tau $. Let $ f\colon \euclid \to \euclid[m] $ be a map. If there exist some $ 0 < \alpha <  1 $ and $ 0 < \Delta < 1 $ such that
	\begin{equation}\label{eq:holderLog}
		\abs{f(x) - f(y)} \lesssim \abs{x-y}^{\alpha} \log (1/\xyAbs) \mFor x ,y \in \euclid \text{ with } \xyAbs \leq \Delta,
	\end{equation}
	then for all sufficiently small $ r > 0 $ and $ x, y \in \euclid$ with $\xyAbs  \leq \Delta $,
	\begin{equation}\label{eq:holderLog-ker}
		\kerProf{r}{s}{\tau}{\abs{f(x)-f(y)}} \gtrsim_{\tau,\alpha} [\log (1/\Phi(r))]^{-\tau} \kerProfAlpha{r}{s}{\tau}{\xyAbs}.
	\end{equation}
	Let $ E \subset \euclid $ be a bounded set. Suppose further that there is some $ 0 < \beta \leq 1 $ such that
	\begin{equation}\label{eq:betaHolder}
		\Abs{f(x)-f(y)} \lesssim \xyAbs^{\beta} \mFor x,y \in E.
	\end{equation}
	Then for all sufficiently small $ r > 0 $,
	\begin{equation}\label{eq:HolderLog-Cap}
		\capProf{r}{s}{\tau}{f(E)} \lesssim_{\tau, \alpha, \beta} [\log (1/\Phi(r))]^{\tau}  \capProfAlpha{r}{s}{m}{E}.
	\end{equation}
\end{lemma}

\begin{proof}
Let $ x, h \in \euclid $ with $ \abs{h} \leq \Delta $. According to \eqref{eq:def-ker-profile-phi}, 
	\begin{equation}\label{eq:ker-HolderLog}
		\begin{aligned}
		\kerProf{r}{s}{\tau}{& \abs{f(x+h)-f(x)}} \\ & = \max_{\Phi(r) \leq u \leq r} u^{-s} \min \left \{ 1, \frac{u^{\tau}}{\abs{f(x+h)-f(x)}^{\tau}}\right \}                                                                                         \\
		& \gtrsim_{\tau} \max_{\Phi(r) \leq u \leq r} u^{- s} \min \left \{ 1, \frac{u^{\tau}}{\abs{h}^{\alpha \tau} [\log(1/\abs{h})]^{\tau} }\right \} \qquad \text{by }\eqref{eq:holderLog}        \\
		& = \max_{\Phi(r) \leq u \leq r} u^{- s} \min\left  \{ 1, \left (  \frac{u}{-\abs{h}^{\alpha} \log \abs{h}}\right )^{\tau} \right \}.
	\end{aligned}
	\end{equation}
	For $ \Phi(r) \leq u \leq r $, write 
	\begin{equation*}
		L(u,h) := \min\left  \{ 1, \left (  \frac{u}{-\abs{h}^{\alpha} \log \abs{h}}\right )^{\tau} \right \} \mAnd M(u,h) := \min\left  \{ 1, \frac{u^{\tau}}{\abs{h}^{\alpha \tau}} \right \}.
	\end{equation*}
	If $ u \geq  -\abs{h}^{\alpha} \log \abs{h} $, then $ L(u, h) = 1 \geq M(u, h) $; If $ u < -\abs{h}^{\alpha} \log \abs{h} $, then $ \Phi(r) \leq - \abs{h}^{\alpha} \log \abs{h} $. Denote $ g(x) := - x^{\alpha}\log x  $ for $ x > 0 $. Note that when $ r $ is small,
	\begin{equation*}
		g\left(\Phi(r)^{2/\alpha}\right) = - \frac{2}{\alpha} \Phi(r)^{2} \log \Phi(r) \leq \Phi(r) \leq g(\abs{h}). 
	\end{equation*}
	Since $ g(x) $ is increasing on $ (0, \exp(-1/\alpha)) $, we have $ \abs{h} \geq \Phi(r)^{2/\alpha} $. Hence
	\begin{equation*}
		L(u, h) = \left (  \frac{u}{-\abs{h}^{\alpha} \log \abs{h}}\right )^{\tau} \geq \left(\frac{\alpha}{2}\right)^{\tau} [\log (1/\Phi(r))]^{-\tau} \frac{u^{\tau}}{\abs{h}^{\alpha \tau}} \gtrsim_{\alpha,\tau} [\log (1/\Phi(r))]^{-\tau} M(u, h).
	\end{equation*}
	This concludes that
	\begin{equation*}
		L(u, h) \gtrsim_{\alpha,\tau} [\log (1/\Phi(r))]^{-\tau} M(u, h) \mFor \Phi(r) \leq u \leq r, \, \abs{h} \leq \Delta.
	\end{equation*}
	Together with \eqref{eq:ker-HolderLog}, we have
	\begin{align*}
		\kerProf{r}{s}{\tau}{\abs{f(x+h)-f(x)}} & \gtrsim_{\tau,\alpha}   [\log (1/\Phi(r))]^{-\tau}  \max_{\Phi(r) \leq u \leq r} u^{- s} \min\left  \{ 1, \frac{u^{\tau}}{\abs{h}^{\alpha \tau}} \right \} \\
	& = [\log (1/\Phi(r))]^{-\tau} \kerProfAlpha{r}{s}{\tau}{\abs{h}},
	\end{align*}
	where the last equality is by taking $ v = u^{1/\alpha} $. This proves \eqref{eq:holderLog-ker}.
	
	Without loss of generality we assume that $ E $ is compact. For $ x, y \in E $ with $ \abs{x-y} \geq \Delta $, \eqref{eq:betaHolder} implies that
	\begin{equation*}
		\Abs{f(x)- f(y)} \lesssim \Abs{E}^{\beta} = \frac{\Abs{E}^{\beta}}{\Delta^{\alpha}} \Delta^{\alpha} \lesssim_{\alpha,\beta} \Delta^{\alpha} \leq \xyAbs^{\alpha},
	\end{equation*}
	thus the calculation in \eqref{eq:ker-holder} shows that for $ x, y \in E $ with $ \xyAbs \geq \Delta $,
	\begin{equation*}
		\kerProf{r}{s}{\tau}{\abs{f(x)-f(y)}} \gtrsim_{\tau,\alpha,\beta} \kerProfAlpha{r}{s}{\tau}{\xyAbs}.
	\end{equation*}
	Together with \eqref{eq:holderLog-ker}, we have for $ x , y \in E $ and sufficiently small $ r > 0 $,
	\begin{equation}\label{eq:ker-set-HolderLog}
		\kerProf{r}{s}{\tau}{\abs{f(x)-f(y)}} \gtrsim_{\tau, \alpha,\beta} [\log (1/\Phi(r))]^{-\tau} \kerProfAlpha{r}{s}{\tau}{\xyAbs} .
	\end{equation}
	Based on \eqref{eq:ker-set-HolderLog}, we apply the arguments in \eqref{eq:Holder-Cap} to get 
	\begin{equation*}\label{eq:fTE<TE}
		\capProf{r}{s}{\tau}{f(E)} \lesssim_{\tau,\alpha,\beta} [\log (1/\Phi(r))]^{\tau} \capProfAlpha{r}{s}{\tau}{E}.
	\end{equation*}
	This finishes the proof.
\end{proof}

To apply \autoref{lem:FBM-Cap}, below we recall L\'evy’s modulus of continuity for the fractional Brownian motion (see e.g., \cite[Chp.~18, Eq.~(3)]{Kahane1985}).

\begin{lemma}\label{lem:LevyCts}
	In \autoref{set:brownian}, let $ E \subset \euclid $ be a bounded set. Then almost surely,
	\begin{equation*}
		\Abs{ \Balph{x} - \Balph{y} } \lesssim \xyAbs^{\alpha/2} \mFor x,y \in E,
	\end{equation*}
	and there exists some $ 0 < \Delta < 1/10 $ such that
	\begin{equation*}
		\Abs{ \Balph{x} - \Balph{y} } \lesssim \xyAbs^{\alpha} \sqrt{\log (1/\xyAbs) } \mFor x,y \in E \text{ with } \xyAbs \leq \Delta.
	\end{equation*}
\end{lemma}

\begin{remark}	
	As is in \cite{Falconer2020,Burrell2021}, the H\"older continuity of the fractional Brownian motion is sufficient for obtaining results about the \tThetaDims. However, for the \tPhiDims, the more precise modulus of continuity in \autoref{lem:LevyCts} seems necessary.
\end{remark}

Now we are ready to prove the key ingredients in the proof of the upper-bound part of \autoref{thm:PhiDim}. They are analogous to \autoref{lem:thetaDim-CoverSumUB}.

\begin{proposition}\label{lem:Phi-CoverSumUB} Let $ \Phi $ be an admissible function.
	\begin{enumerate}[(i)]
		\item\label{itm:phi-UB-affine} In \autoref{set:selfaffine}, let $ E \subset \Sigma $ be compact and $ 0 \leq s \leq d $. Let $ \bfa \in \euclid[dm] $. If for $ 0 < r \leq 1 $ there exist $ \mu \in \calP(E) $ and $ \gamma > 0 $ such that
		      \begin{equation*}
			      \int \kerPhi{r}{\xWy} \, d\mu(y) \geq \gamma \mFor \text{all } x\in E
		      \end{equation*}
		      then for all sufficiently small $ r > 0$,
		      \begin{equation*}
			      \sumPhi{r}{\pia (E) } \lesssim  \frac{\log(1/\Phi(r))}{\gamma}.
		      \end{equation*}
		      In particular,
		      \begin{equation}\label{eq:UB-Selfaffine}
			      \sumPhi{r}{\pia(E)} \lesssim_{d,\bfa} \log(1/\Phi(r)) \capPhi{r}{E}.
		      \end{equation}

		\item\label{itm:phi-UB-ortho} In \autoref{set:orthogon}, let $ F \subset \euclid[m]$  be compact and $ 0 \leq s \leq m $. If for $ 0 < r \leq 1 $ there exist $ \mu \in \calP(F) $ and $ \gamma > 0 $ such that
		      \begin{equation}\label{eq:potential-phi-LB}
			      \int \kerProf{r}{s}{m}{\xyAbs} \, d\mu(y) \geq \gamma \mFor \text{all } x \in F,
		      \end{equation}
		      then for all sufficiently small  $ r > 0 $,
		      \begin{equation}\label{eq:UB-Ortogon-PDP}
			      \sumPhi{r}{F} \lesssim  \frac{ \log(1/\Phi(r)) }{\gamma}.
		      \end{equation}
		      In particular, for $ E\subset \euclid[d] $ and $ V\in G(d,m) $,
		      \begin{equation}\label{eq:UB-Orthogon}
			      \sumPhi{r}{\piV E } \lesssim \log(1/\Phi(r)) \, \capProf{r}{s}{m}{E}.
		      \end{equation}

		\item\label{itm:phi-UB-brown} In \autoref{set:brownian}, let $ E \subset \euclid[d] $ be compact and $ 0 \leq s \leq m $. Then almost surely for all sufficiently small $ r > 0 $,
		      \begin{equation}\label{eq:UB-Brownian}
			      \sumProfile{r}{s}{\Balph{E}} \lesssim [\log(1/\Phi(r))]^{m + 1} \capAlphaPhi{r^{1/\alpha}}{\alpha s}{\alpha m}{E}.
		      \end{equation}
	\end{enumerate}
\end{proposition}

\begin{proof}
	 \ref{itm:phi-UB-affine} follows from a similar proof of \autoref{lem:thetaDim-CoverSumUB}.

	Next we prove \ref{itm:phi-UB-ortho} by adapting some ideas of \cite{BurrellEtAl2021} but considering a different kernel $ \kerProf{r}{s}{m}{\xyAbs} $.
	Write $ D := \lceil \log (\abs{F} / \Phi(r)) \rceil $. Let $ x\in F $. Then $ F \subset B(x, \abs{F}) \subset B(x, \exp(D)\, \Phi(r)) $.  For simplicity, define
	\begin{equation*}
		\delta_{k} = \exp(k) \Phi(r) \mFor k = 0, 1, 2,\ldots
	\end{equation*}
	By convention we set $ \delta_{-1} := \Phi(r) $.

	Since $ \kerProf{r}{s}{m}{\xyAbs} = \Phi(r)^{-s} $ for $ y \in B(x,\Phi(r)) $ and $ \Delta \mapsto \kerProf{r}{s}{m}{\Delta} $ is non-increasing, by \eqref{eq:potential-phi-LB},
	\begin{align*}
		\gamma & \leq\int \kerProf{r}{s}{m}{\xyAbs} \, d \mu(y) \\&  = \int_{ B(x, \delta_{0}) } \kerProf{r}{s}{m}{ \xyAbs } \, d \mu(y) + \sum_{k=1}^{D} \int_{B(x, \delta_{k}) \setminus B(x, \delta_{k-1})} \kerProf{r}{s}{m}{ \xyAbs } \, d \mu(y) \\& \leq \sum_{k=0}^{D} \kerProf{r}{s}{m}{\delta_{k-1}} \mu(B(x, \delta_{k})).
	\end{align*}
	Hence for each $ x \in F $, there exists some integer $ k(x) \in [0, D] $ such that
	\begin{equation}\label{eq:tmp1}
		\kerProf{r}{s}{m}{\delta_{k(x)-1}} \mu(B(x,\delta_{k(x)})) \geq \frac{\gamma}{D+1}.
	\end{equation}
	By \eqref{eq:def-ker-profile-phi}, we can find some $ u(x) \in [\Phi(r), r] $ such that
	\begin{equation}\label{eq:tmp2}
		u(x) ^{-s} \kerGeoCount{u(x)}{m}{\delta_{k(x)-1}} = \kerProf{r}{s}{m}{\delta_{k(x)-1}}.
	\end{equation}
	Since
	\begin{equation*}
		\kerGeoCount{u(x)}{m}{\delta_{k(x)}} \geq \exp(-m) \kerGeoCount{u(x)}{m}{\delta_{k(x)-1}} \gtrsim_{m} \kerGeoCount{u(x)}{m}{\delta_{k(x)-1}},
	\end{equation*}
	it follows from \eqref{eq:tmp1} and \eqref{eq:tmp2} that
	\begin{equation*}
		u(x) ^{-s} \kerGeoCount{u(x)}{m}{\delta_{k(x)}} \mu(B(x, \delta_{k(x)})) \gtrsim_{m} \frac{\gamma}{D+1}.
	\end{equation*}
	By a rearrangement,
	\begin{equation}\label{eq:UB-us-counting}
		\frac{ u(x)^{s}}{\kerGeoCount{u(x)}{m}{\delta_{k(x)}}} \lesssim_{m} \frac{D+1}{\gamma} \mu(B(x,\delta_{k(x)})).
	\end{equation}


	Let $ \scrB :=  \{ B(x, \delta_{k(x)})\}_{x\in F} $. Then $ \scrB $ is a cover of $ F $. For each $ B = B(x, \delta_{k(x)}) \in \scrB $, write $ \delta_{B} := \delta_{k(x)} $ and $ u_{B} := u(x) \in [\Phi(r), r]$. Then \eqref{eq:UB-us-counting} becomes
	\begin{equation}\label{eq:thick-balls}
		\frac{ u_{B}^{s}}{\kerGeoCount{u_{B}}{m}{\delta_{B}}} \lesssim_{m} \frac{D+1}{\gamma} \mu(B) \mFor B \in \scrB.
	\end{equation}
	By \autoref{lem:kerGeo-counting}, for each $ B \in \scrB $ there is a cover $ \Gamma_{B} $ consisting of sets with diameter $ u_{B} $ such that
	\begin{equation}\label{eq:GeoCount}
		\# \Gamma_{B}  \lesssim_{m} \frac{1}{\kerGeoCount{u_{B}}{m}{\delta_{B}}}.
	\end{equation}
	By Besicovitch covering theorem (see \cite[Theorem 2.7]{Mattila1995}), there are $ \scrB_{1} , \ldots, \scrB_{c} \subset \scrB $ covering $ F $ such that each $ \scrB_{i} $ is disjoint, where $ c $ only depends on $ m $, that is,
	\begin{equation*}
		F \subset \Union_{i=1}^{c} \Union_{B\in \scrB_{i}} B
	\end{equation*}
	and
	\begin{equation*}
		B \intersect B' = \emptyset \mFor B, B' \in \scrB_{i},\, B \neq B',\, i = 1, \ldots, c.
	\end{equation*}
	Then the collection
	$ 	\Union_{i=1}^{c} \Union_{B\in \scrB_{i}} \Gamma_{B} $
	covers $ F $ and consists of sets with diameters in $ [\Phi(r), r] $. Hence
	\begin{align*}
		\sumProfile{r}{s}{F} & \leq \sum_{i=1}^{c} \sum_{B\in \scrB_{i}} \sum_{U\in \Gamma_{B}} \abs{U}^{s} \\ & = \sum_{i=1}^{c} \sum_{B\in \scrB_{i}}  \#\Gamma_{B} \cdot u_{B}^{s} \\ & \lesssim_{m} \sum_{i=1}^{c} \sum_{B\in \scrB_{i}}  \frac{u_{B}^{s}}{\kerGeoCount{u_{B}}{m}{\delta_{B}}} & \text{by } \eqref{eq:GeoCount} \\
		& \lesssim_{m} \sum_{i=1}^{c} \sum_{B\in \scrB_{i}} \frac{D+1}{\gamma} \mu(B) & \text{by }\eqref{eq:thick-balls} \\ & \lesssim_{m} \frac{D+1}{\gamma} & \text{by each } \scrB_{i} \text{ disjoint} \\
		& \lesssim \frac{\log (1/\Phi(r))}{\gamma}
	\end{align*}
	when $ r $ is small. This proves \eqref{eq:UB-Ortogon-PDP}. 
	
	Since \autoref{lem:equil-meas} gives an equilibrium measure $ \mu \in\calP(F) $ satisfying \eqref{eq:potential-phi-LB} with $ \gamma = 1/\capProf{r}{s}{m}{F} $, it follows from \eqref{eq:UB-Ortogon-PDP} that
	\begin{equation}\label{eq:UB-Orthogo-Image}
		\sumProfile{r}{s}{F} \lesssim \log(1/\Phi(r)) \capProf{r}{s}{m}{F}.
	\end{equation}
	Since $ \abs{\piV x - \piV y} \leq \abs{x - y} $ for each orthogonal projection $ \piV $, applying \autoref{lem:HolderCap} with $ \alpha = 1 $ shows
	\begin{equation}\label{eq:LipCap}
		\capProf{r}{s}{m}{\piV E } \lesssim \capProf{r}{s}{m}{E}.
	\end{equation}
	Applying \eqref{eq:UB-Orthogo-Image} with $ F = \piV E $, we obtain \eqref{eq:UB-Orthogon} from \eqref{eq:LipCap}.

	Finally we move to \ref{itm:phi-UB-brown}. By \autoref{lem:LevyCts} and \autoref{lem:FBM-Cap}, we have almost surely that
	\begin{equation*}
		\capProf{r}{s}{m}{\Balph{E}} \lesssim_{m,\alpha} [\log(1/\Phi(r))]^{m} \capProfAlpha{r}{s}{m}{E}.
	\end{equation*}
	Hence applying \eqref{eq:UB-Orthogo-Image} with $ F = \Balph{E} $ gives that almost surely, 
	\begin{equation*}
		\sumProfile{r}{s}{\Balph{E}} \lesssim [\log(1/\Phi(r))]^{m+1} \capProfAlpha{r}{s}{m}{E}.
	\end{equation*}
	This completes the proof.
\end{proof}

\subsection{Lower bound}

We begin with an analog of \autoref{lem:thetaDim-CoverSumLB}.
For $ 0 \leq s \leq d $ and $ r > 0 $, define
\begin{equation}\label{eq:def-phi-truncatedKernel}
\kerPsiPhi{r}{s}{\Delta} \coloneqq \begin{cases}
		\Phi(r)^{-s} & \Delta \leq \Phi(r)     \\
		\Delta^{-s}  & \Phi(r) < \Delta \leq r \\
		0            & \Delta > r
	\end{cases} \quad \text{ for } \Delta \geq 0.
\end{equation}

\begin{lemma}\label{lem:phiDim-CoverSumLB}
	Let $ \Phi $ be an admissible function. Let $ 0 \leq s \leq d $, $ 0 < r \leq 1 $, and $ E \subset \euclid[d]$ be a non-empty compact set. If there exist $ \mu \in \calP(E) $, a Borel subset $ F \subset E $, and $ \gamma > 0 $ such that
	\begin{equation*}
		\int \kerPsiPhi{r}{s}{\xyAbs} \, d\mu(y) \leq \gamma \mFor \text{all } x\in F,
	\end{equation*}
	then
	\begin{equation*}
		\sumPhi{r}{E} \geq  \frac{\mu(F)}{\gamma}.
	\end{equation*}
\end{lemma}

\autoref{lem:phiDim-CoverSumLB} follows from a similar proof of \autoref{lem:thetaDim-CoverSumLB}. Below we provide two lemmas showing that there are some appropriate transversalities in \autoref{set:orthogon} and \autoref{set:brownian}, which are the analogs of \autoref{lem:paramMeasure}.

\begin{lemma}[{\cite[Lemma 3.11]{Mattila1995}}]\label{lem:trans-orthon}
	In \autoref{set:orthogon}, let $ x , y \in \euclid[d] $ and $ r > 0 $. Then
	\begin{equation}\label{eq:trans-orthogon}
		\gamma_{d,m}\{ V \in G(d, m) \colon \xyPiV \leq r \} \lesssim_{d,m} \kerGeoCount{r}{m}{\xyAbs}
	\end{equation}
	where $ \kerGeoCount{r}{m}{\xyAbs} $ is as in \eqref{eq:def-ker-geoCounting}.
\end{lemma}

\begin{lemma} \label{lem:trans-FBM}
	In \autoref{set:brownian}, let $ x,y \in \euclid $ and $ r > 0 $. Then
	\begin{equation}\label{eq:trans-FBM}
		\bbP\{ \xyBalpha \leq r \} \lesssim_{m} \kerGeoCount{r^{1/\alpha}}{\alpha m}{\xyAbs}.
	\end{equation}
\end{lemma}
\begin{proof}
	By \eqref{eq:FBM-dist},
	\begin{align*}
		\bbP \{ \xyBalpha \leq r \} & \leq \bbP \{ \Abs{ B_{\alpha, i}(x) - B_{\alpha, i}(y) } \leq r \text{ for all }  1\leq i \leq m \} \\
		& = \left ( \frac{1}{\sqrt{2\pi} \xyAbs^{\alpha} } \int_{\abs{t} \leq r}  \exp \left( - \frac{t^{2} }{ 2\xyAbs^{2\alpha} } \right)  \, dt \right )^{m} \\
		& \leq \left ( \frac{1}{ \xyAbs^{\alpha} } \int_{\abs{t} \leq r} 1  \, dt \right )^{m} \\
		& = 2^{m} \frac{r^{m}}{\xyAbs^{\alpha m}}\\
		& \lesssim_{m} \left(  \frac{r^{1/\alpha}}{\xyAbs} \right)^{\alpha m}.
	\end{align*}
	Since $ \bbP(A) \leq 1 $ for all events $ A \subset \Omega $, we have
	\begin{equation*}
		\bbP \{ \xyBalpha \leq r \} \lesssim_{m} \min \left\{ 1, \left(  \frac{r^{1/\alpha}}{\xyAbs} \right)^{\alpha m} \right\} = \kerGeoCount{r^{1/\alpha}}{\alpha m}{\xyAbs}.
	\end{equation*}
	This finishes the proof.
\end{proof}

As an analog of \autoref{lem:thetaDim-paraMeasure}, the following lemma reveals a unified computational scheme for the integrals over parameters in various contexts.

\begin{proposition}\label{lem:Phi-TransIntegral} Let $ \Phi $ be an admissible function and $ 0 < r \leq 1 $.

	\begin{enumerate}[(i)]
		\item In \autoref{set:selfaffine}, assume $ \norm{T_{j}} < 1/2 $ for $ 1 \leq j \leq m $. Let $ 0 \leq s \leq d $ and $ x, y \in \Sigma $. Then
		      \begin{equation*}
			      \int_{B_{\rho}} \kerPsiPhi{r}{s}{\DeltaPiXY} \, d\mathbf{a} \lesssim_{\rho, d} \log(1/\Phi(r)) \kerPhi{r}{\xWy}
		      \end{equation*}
		      where  $ B_{\rho} $ denotes the closed ball in $ \euclid[dm] $ centered at $ 0 $ with radius $ \rho > 0 $.

		\item In \autoref{set:orthogon}, let $  0 \leq s \leq m $ and $ x, y \in \euclid[d] $. Then
		      \begin{equation*}
			      \int_{G(d,m)} \kerPsiPhi{r}{s}{\xyPiV} \, d\gamma_{d,m}(V) \lesssim_{d,m} \log(1/\Phi(r)) \kerProf{r}{s}{m}{\xyAbs}.
		      \end{equation*}

		\item\label{itm:LB-brown} In \autoref{set:brownian}, let $  0 \leq s \leq m $ and $ x, y \in \euclid[d] $. Then
		      \begin{equation*}
			      \int_{\Omega} \kerPsiPhi{r}{s}{\xyBalpha} \, d\bbP(\omega) \lesssim_{m} \log(1/\Phi(r)) \kerAlphaPhi{r^{1/\alpha}}{\alpha s}{\alpha m}{\xyAbs}.
		      \end{equation*}
	\end{enumerate}
\end{proposition}

\begin{proof}
	We begin with a general computational scheme assuming the abstract transversality \eqref{eq:tmp-transversality}. Then the proof is completed by substituting \eqref{eq:tmp-transversality} with the corresponding transversality in different settings.

	Let $ (\Lambda, \nu) $ be a measure space and $ \lambda \mapsto \Delta_{\lambda} $ be a measurable function from $ \Lambda $ to $ (0, +\infty)$. Suppose
	\begin{equation}\label{eq:tmp-transversality}
		\nu\{\lambda  \colon \Delta_{\lambda} \leq r \} \lesssim K_{r} \mFor r > 0,
	\end{equation}
	where $ r \mapsto K_{r} $ is a measurable function. According to \eqref{eq:def-phi-truncatedKernel},
	\begin{align*}
		  \int_{\Lambda} & \kerPsiPhi{r}{s}{\Delta_{\lambda}} \, d\nu(\lambda)                                                                                                                                       \\ & = \int_{\{\lambda  \colon \Delta_{\lambda} \leq \Phi(r) \}} \Phi(r)^{-s} \, d\nu(\lambda) + \int_{\{ \lambda  \colon \Phi(r) < \Delta_{\lambda} \leq r \}} \Delta_{\lambda}^{-s} \, d\nu(\lambda)  \\
		 & = \Phi(r)^{-s} \nu \{\lambda  \colon \Delta_{\lambda} \leq \Phi(r) \} +  \int_{0}^{\infty} \nu \{\lambda  \colon  \Phi(r) < \Delta_{\lambda} \leq r,\; \Delta_{\lambda}^{-s} \geq t \} \, dt \\
		 & = \Phi(r)^{-s} \nu\{\lambda  \colon \Delta_{\lambda} \leq \Phi(r) \} + \int_{0}^{\Phi(r)^{-s}} \nu \{\lambda \colon \Phi(r) <\Delta_{\lambda} \leq \min\{r, t^{-1/s}\} \} \, dt              \\
		 & =  \int_{0}^{\Phi(r)^{-s}} \nu \{\lambda \colon \Delta_{\lambda} \leq \min\{r, t^{-1/s}\} \} \, dt                                                                                                      \\
		 & =  \int_{0}^{r^{-s}} \nu \{\lambda  \colon  \Delta_{\lambda} \leq r \} \, dt +  \int_{r^{-s}}^{\Phi(r)^{-s}} \nu \{\lambda \colon  \Delta_{\lambda} \leq t^{-1/s} \} \, dt                    \\
		 & = r^{-s} \nu \{\lambda \colon  \Delta_{\lambda} \leq r \} + \int_{r^{-s}}^{\Phi(r)^{-s}} \nu \{\lambda \colon  \Delta_{\lambda} \leq t^{-1/s} \} \, dt.
	\end{align*}
	By changing variable with $ u = t^{-1/s} $,
	\begin{equation}\label{eq:abstract-computation}
		\begin{aligned}
			  \int_{\Lambda} & \kerPsiPhi{r}{s}{\Delta_{\lambda}}  \, d\nu(\lambda)                                                                                            \\
			 & =  r^{-s} \nu \{\lambda \colon  \Delta_{\lambda} \leq r \} + s \int_{\Phi(r)}^{r} u^{-(s+1)} \nu \{\lambda \colon  \Delta_{\lambda} \leq u \} \, du \\
			 & \lesssim r^{-s} K_{r} + s \int_{\Phi(r)}^{r} u^{-(s+1)} K_{u} \, du                                                                                                     & \text{by } \eqref{eq:tmp-transversality} \\
			 & =   \left (1 + s \int_{\Phi(r)}^{r} u^{-1}\, du \right )  \max_{\Phi(r) \leq u \leq r} u^{-s} K_{u}                                                                     \\
			 & \leq (s+1) \log(1/\Phi(r)) \left ( \max_{\Phi(r) \leq u \leq r} u^{-s} K_{u} \right ),
		\end{aligned}
	\end{equation}
	where the last inequality follows from $  \int_{\Phi(r)}^{r} u^{-1}\, du \leq \log (1/\Phi(r))$. 
	
	Finally by replacing \eqref{eq:tmp-transversality} with \autoref{lem:paramMeasure}, \autoref{lem:trans-orthon}, and \autoref{lem:trans-FBM} respectively, we finish the proof by \eqref{eq:abstract-computation}.
\end{proof}

\subsection{Proof of \autoref{thm:PhiDim}}
\begin{proof}[Proof of \autoref{thm:PhiDim}]
	Based on \autoref{lem:Phi-CoverSumUB} and \autoref{lem:Phi-TransIntegral}, the statements of different settings in \autoref{thm:PhiDim} result from similar arguments. Hence to avoid repetitions while maintaining clarity, we exemplify the arguments by showing \eqref{eq:GenThm-ortho-UB} and \eqref{eq:GenThm-brown}. Without loss of generality, we assume that $ E $ is compact.
	
	For \eqref{eq:GenThm-ortho-UB}, we show $ \lDim{\Phi} \piV E \leq \lDim{\Phi}^{m} E $ while the proof of $ \uDim{\Phi} \piV E \leq \uDim{\Phi}^{m} E $ is similar. Let $ s > t > \lDim{\Phi}^{m} E $. Then
	\begin{equation}\label{eq:C-t-tozero}
		\liminf_{r\to 0} \capProf{r}{t}{m}{E} = 0.
	\end{equation}
	By \eqref{eq:GenThm-Cond}, there is some $ A > 0 $ such that
	\begin{equation*}
		r^{s-t} \log (1/\Phi(r)) \leq A \mFor r > 0.
	\end{equation*}
	Hence by \ref{itm:phi-UB-ortho} of \autoref{lem:Phi-CoverSumUB} and \eqref{eq:GenDim-justify},
	\begin{equation*}
		\sumPhi[s]{r}{\piV E} \lesssim \log(1/\Phi(r)) \capProf{r}{s}{m}{E} \leq r^{s-t} \log(1/\Phi(r)) \capProf{r}{t}{m}{E}  \leq A \, \capProf{r}{t}{m}{E}.
	\end{equation*}
	By \eqref{eq:C-t-tozero}, taking $ \liminf_{r\to 0} $ on both sides of the above inequality implies 
	\begin{equation*}
		\liminf_{r\to 0} \sumPhi{r}{\piV E} = 0. 
	\end{equation*}	
	This shows $ \lDim{\Phi} \piV E \leq s $ by \autoref{lem:char-PhiDim}. Letting $ s \to \lDim{\Phi}^{m} E $ gives
	\begin{equation*}
		 \lDim{\Phi} \piV E \leq \lDim{\Phi}^{m} E.
	\end{equation*}
	
	Next we prove \eqref{eq:GenThm-brown} by showing that almost surely $ \uDim{\Phi} B_{\alpha}(E) = \frac{1}{\alpha}\uDim{\Phi_{\alpha}}^{\alpha m} E $ while the proof for almost surely $ \lDim{\Phi} B_{\alpha}(E) = \frac{1}{\alpha}\lDim{\Phi_{\alpha}}^{\alpha m} E $ is similar. Let $ s > t > \frac{1}{\alpha}\uDim{\Phi_{\alpha}}^{\alpha m} E $. Then
	\begin{equation}\label{eq:brown-UB-toZero}
	\limsup_{r\to 0} \capProfAlpha{r}{t}{m}{E} =	\limsup_{r\to 0} \capAlphaPhi{r}{\alpha t}{\alpha m}{E} = 0.
	\end{equation}
	By \eqref{eq:GenThm-Cond}, there is some $ A > 0 $ such that
	\begin{equation}\label{eq:A-UB}
		r^{(s-t)/(m+1)} \log (1/\Phi(r)) \leq A \mFor r > 0.
	\end{equation}
	By \ref{itm:phi-UB-brown} of \autoref{lem:Phi-CoverSumUB}, almost surely,
	\begin{align*}
		\sumPhi[s]{r}{\Balph{E}} & \lesssim [\log(1/\Phi(r))]^{m+1} \capProfAlpha{r}{s}{m}{E} \\& \leq r^{(s-t)} [\log(1/\Phi(r))]^{m+1} \capProfAlpha{r}{t}{m}{E}  & \text{by } \eqref{eq:GenDim-justify}\\& \leq A^{m+1} \, \capAlphaPhi{r^{1/\alpha}}{\alpha t}{\alpha m}{E} & \text{by }\eqref{eq:A-UB}.
	\end{align*}
	By taking $ \limsup_{r\to 0} $ on both sides, it follows from \eqref{eq:brown-UB-toZero} that almost surely,
	\begin{equation*}
		\limsup_{r\to 0} \sumPhi[s]{r}{\Balph{E}} = 0. 
	\end{equation*}	
	Then $ \uDim{\Phi}\Balph{E} \leq s $ by \autoref{lem:char-PhiDim}. Letting $ s \to \frac{1}{\alpha} \uDim{\Phi_{\alpha}}^{\alpha m} E $ gives
	\begin{equation*}
		\uDim{\Phi}\Balph{E} \leq \frac{1}{\alpha}\uDim{\Phi_{\alpha}}^{\alpha m} E.
	\end{equation*}
	Hence it suffices to prove that almost surely
	\begin{equation}\label{eq:brown-LB}
		 \uDim{\Phi}\Balph{E} \geq \frac{1}{\alpha}\uDim{\Phi_{\alpha}}^{\alpha m} E.
	\end{equation}
	Suppose $ \frac{1}{\alpha}\uDim{\Phi_{\alpha}}^{\alpha m} E > 0 $, otherwise \eqref{eq:brown-LB} holds trivially. Let $ t < s < \frac{1}{\alpha}\uDim{\Phi_{\alpha}}^{\alpha m} E  $. Take a sequence $ (r_{k}) $ tending to $ 0 $ such that $ 0 < r_{k} \leq 2^{-k} $ and 
	\begin{equation}\label{eq:choose-rk}
		\limsup_{k\to\infty} \capAlphaPhi{r_{k}^{1/\alpha}}{\alpha s}{\alpha m}{E} = \limsup_{r\to 0} \capAlphaPhi{r}{\alpha s}{\alpha m}{E} > 0.
	\end{equation}
	By \autoref{lem:equil-meas}, for each $ k \in \bbN $ there is an equilibrium measure $ \mu_{k} $ on $ E $ for the kernel $ \kerProfAlpha{r_{k}}{s}{m}{\xyAbs} $. Write
	\begin{equation*}
		\gamma_{k} \coloneqq \frac{1}{ \capProfAlpha{r_{k}}{s}{m}{E} } = \iint \kerProfAlpha{r_{k}}{s}{m}{\xyAbs} \, d\mu_{k}(x) d\mu_{k}(y).
	\end{equation*}
	By \ref{itm:LB-brown} of \autoref{lem:Phi-TransIntegral},
	\begin{equation}\label{eq:brown-UB-Integral}
	\iint \int_{\Omega} \kerPsiPhi{r_{k}}{s}{\xyBalpha} \, d\bbP(\omega) \; d\mu_{k}(x) d\mu_{k}(y) \lesssim \log(1/\Phi(r_{k})) \gamma_{k}.
	\end{equation}
	Set $ \varepsilon := s - t $. By \eqref{eq:GenThm-Cond}, there is some $ A > 0 $ such that
	\begin{equation}\label{eq:brown-A-PhiEps}
		r^{\varepsilon/2} \log (1/\Phi(r)) \leq A \mFor r > 0.
	\end{equation}
	Then summing \eqref{eq:brown-UB-Integral} over $ k\in\bbN $ and using Fubini's theorem lead to
	\begin{align*}
		\int_{\Omega}  & \sum_{k=1}^{\infty} \left (  r_{k}^{\varepsilon} \gamma_{k}^{-1} \iint \kerPsiPhi{r_{k}}{s}{\xyBalpha} \, d\mu_{k}(x) d\mu_{k}(y)  \right) d\bbP(\omega)  \\
		& \lesssim \sum_{k=1}^{\infty} \log(1/\Phi(r_{k})) r_{k}^{\varepsilon} & \text{by } \eqref{eq:brown-UB-Integral}  \\
		& \leq A \sum_{k=1}^{\infty} r_{k}^{\varepsilon/2}  & \text{by }\eqref{eq:brown-A-PhiEps} \\
		& \leq A \sum_{k=1}^{\infty} 2^{-k\varepsilon/2} < \infty & \text{by } r_{k}\leq 2^{-k}. 
	\end{align*}
	Hence almost surely there exists $ M > 0 $ such that
	\begin{equation*}
		\iint \kerPsiPhi{r_{k}}{s}{\abs{u-v}} \, dB_{\alpha}\mu_{k}(v) dB_{\alpha}\mu_{k}(u) \leq M \gamma_{k} r_{k}^{-\varepsilon} \mFor \text{all } k \in \bbN.
	\end{equation*}
	Then for each $ k\in \bbN  $ there is some $ F_{k} \subset \Balph{E} $ such that $ B_{\alpha}\mu_{k}(F_{k}) \geq 1/2 $ and
	\begin{equation*}
		\int \kerPsiPhi{r_{k}}{s}{\abs{u-v}} \, dB_{\alpha}\mu_{k}(v) \leq 2 M \gamma_{k} r_{k}^{-\varepsilon} \mFor \text{all } u \in F_{k}.
	\end{equation*}
	It follows from \autoref{lem:phiDim-CoverSumLB} that for each $ k\in \bbN $,
	\begin{equation}\label{eq:Sum>=Cap}
		\sumPhi{r_{k}}{\Balph{E}} \geq \frac{1}{2}  (2M\gamma_{k} r_{k}^{-\varepsilon})^{-1} \gtrsim r_{k}^{\varepsilon}\gamma_{k}^{-1} = r_{k}^{\varepsilon} \, \capAlphaPhi{r_{k}^{1/\alpha}}{\alpha s}{\alpha m}{E}.
	\end{equation}
	Finally, almost surely we have
	\begin{align*}
	 	\limsup_{r\to 0} \sumPhi[t]{r}{\Balph{E}} & \geq \limsup_{k\to\infty} \sumPhi[t]{r_{k}}{\Balph{E}} \\
	 	& \geq \limsup_{k\to\infty} r_{k}^{-(s-t)} \sumPhi[s]{r_{k}}{\Balph{E}} & \text{by } \eqref{eq:sumPhi-compare} \\
	 	& \gtrsim \limsup_{k\to\infty} r_{k}^{-(s-t)} r_{k}^{\varepsilon} \, \capProfAlpha{r_{k}}{s}{m}{E} & \text{by }\eqref{eq:Sum>=Cap} \\
	 	& = \limsup_{k\to\infty} \capProfAlpha{r_{k}}{s}{m}{E} & \text{by } \varepsilon = s-t \\
	 	& > 0 & \text{by }\eqref{eq:choose-rk}
	\end{align*}
	This shows that almost surely,
	\begin{equation*}
		\uDim{\Phi} \Balph{E} \geq t.
	\end{equation*}
	Letting $ t \to \frac{1}{\alpha}\uDim{\Phi_{\alpha}}^{\alpha m} E  $ gives \eqref{eq:brown-LB}.
\end{proof}

\section{Final remarks}
In the section we give a few remarks.

In our main theorems, the assumption that $ \norm{T_{j}} < 1/2 $ for $ 1 \leq j \leq m $ can be weaken to $ \max_{i\neq j} (\norm{T_{i}} + \norm{T_{j}}) < 1 $. Indeed the first assumption is only used to guarantee the self-affine transversality in \autoref{lem:paramMeasure}. As pointed out in \cite[Proposition 9.4.1]{BaranyEtAl2022}, the second assumption is sufficient for the self\nobreakdash-affine transversality. 

In \cite[Definition 2.7]{Banaji2020}, the admissibility of $\Phi$  is assumed in the definitions of the \tPhiDims\ in some general metric spaces. However, in \autoref{thm:PhiDim} concerning the \tPhiDims\ in $\euclid$, we only require that $ \Phi $ is monotone and satisfies $0 < \Phi(r) \leq r$ instead of the admissibility.

There is no obstruction in adapting the arguments in \cite[Section 9]{FengEtAl2022} to estimate the Hausdorff dimensions of the exceptional sets for the \tPhiDims. For example, below we give one such result.

\begin{proposition}
	In \autoref{set:selfaffine}, assume $ \norm{T_{j}} < 1/2 $ for $ 1 \leq j \leq m $. Let $ \Phi $ be an admissible function satisfying \eqref{eq:GenThm-Cond}. Then for $ 0 < \delta < d $,
	\begin{equation*}\label{eq:except-lDim}
		\dimH \{ \bfa \in\euclid[dm] \colon \lDim{\Phi} \pia (E) < \lDim{C,\Phi}  E - \delta \} \leq dm - \delta,
	\end{equation*}
	and
	\begin{equation*}\label{eq:except-uDim}
		\dimH \{ \bfa \in\euclid[dm] \colon \uDim{\Phi} \pia(E) < \uDim{C,\Phi} E - \delta \} \leq dm - \delta .
	\end{equation*}
\end{proposition}

Inspired by \cite[Corollary 6.4]{BurrellEtAl2021} and \cite[Theorem 6.1]{Banaji2020}, we can deduce an interesting corollary from \autoref{thm:PhiDim} by proving the analogs of the corollaries in \cite[Section 6]{BurrellEtAl2021}.

\begin{corollary}\label{coro:ValProjBox}
	Let $ E \subset \euclid $ be a bounded set. Suppose there is a family of admissible functions $ \{\Psi_{s}\} $ such that $ \lDim{\Psi_{s}} E = s $ and $ \Psi_{s} $ satisfies \eqref{eq:GenThm-Cond} for $ s\in [\dimH E, \lDim{B} E] $. Then $ \lDim{B} \piV E = m $ for $ \gamma_{d,m} $-a.e.~$ V\in G(d, m) $ if and only if $ \dimH E \geq m $.
	
	A similar result holds for the upper dimensions replacing $ \lDim{\Psi} E $ and $ \lDim{B} E $ with $ \uDim{\Psi} E $ and $ \uDim{B} E $, respectively.
\end{corollary}

\noindent\textbf{Acknowledgements.} The author is grateful to De-Jun Feng for the insightful discussions and his valuable comments. Additionally, the author would like to thank Jian-Ci Xiao and Yu-Feng Wu for their helpful suggestions on the preparation of this article. 

\label{sec:ref}


\begin{thebibliography}{22}
	
	\bibitem[Banaji(2020)]{Banaji2020}
	A.~Banaji.
	\newblock Generalised intermediate dimensions.
	\newblock {\em arXiv preprint
		\href{http://arxiv.org/abs/2011.08613}{arXiv:2011.08613}}, 2020.
	
	\bibitem[Banaji and Kolossváry(2021)]{BanajiKolossvary2021}
	A.~Banaji and I.~Kolossváry.
	\newblock Intermediate dimensions of Bedford-McMullen carpets with applications
	to Lipschitz equivalence.
	\newblock {\em arXiv preprint
		\href{http://arxiv.org/abs/2111.05625}{arXiv:2111.05625}}, 2021.
	
	\bibitem[B\'{a}r\'{a}ny et~al.(2022)B\'{a}r\'{a}ny, Simon, and
	Solomyak]{BaranyEtAl2022}
	B.~B\'{a}r\'{a}ny, K.~Simon, and B.~Solomyak.
	\newblock \emph{Self-similar and self-affine sets and measures}.
	\newblock 2022.
	\newblock Book in preparation.
	
	\bibitem[Burrell(2021)]{Burrell2021}
	S.~A. Burrell.
	\newblock Dimensions of Fractional Brownian Images.
	\newblock \emph{Journal of Theoretical Probability}, 2021.
	
	\bibitem[Burrell et~al.(2021)Burrell, Falconer, and Fraser]{BurrellEtAl2021}
	S.~A. Burrell, K.~J. Falconer, and J.~M. Fraser.
	\newblock Projection theorems for intermediate dimensions.
	\newblock \emph{J. Fractal Geom.}, 8\penalty0 (2):\penalty0 95--116, 2021.
	
	\bibitem[Falconer(1988)]{Falconer1988}
	K.~J. Falconer.
	\newblock The {H}ausdorff dimension of self-affine fractals.
	\newblock \emph{Math. Proc. Cambridge Philos. Soc.}, 103\penalty0 (2):\penalty0
	339--350, 1988.
	
	\bibitem[Falconer(2003)]{Falconer2003}
	K.~J. Falconer.
	\newblock \emph{Fractal geometry: Mathematical foundations and applications}.
	\newblock John Wiley \& Sons, Inc., Hoboken, NJ, second edition, 2003.
	
	\bibitem[Falconer(2020)]{Falconer2020}
	K.~J. Falconer.
	\newblock A capacity approach to box and packing dimensions of projections and
	other images.
	\newblock In \emph{Analysis, probability and mathematical physics on fractals},
	volume~5 of \emph{Fractals Dyn. Math. Sci. Arts Theory Appl.}, pages 1--19.
	World Sci. Publ., Hackensack, NJ, 2020.
	
	\bibitem[Falconer(2021{\natexlab{a}})]{Falconer2021}
	K.~J. Falconer.
	\newblock A capacity approach to box and packing dimensions of projections of
	sets and exceptional directions.
	\newblock \emph{J. Fractal Geom.}, 8\penalty0 (1):\penalty0 1--26,
	2021{\natexlab{a}}.
	
	\bibitem[Falconer(2021{\natexlab{b}})]{Falconer2021a}
	K.~J. Falconer.
	\newblock Intermediate dimensions: a survey.
	\newblock In \emph{Thermodynamic formalism}, volume 2290 of \emph{Lecture Notes
		in Math.}, pages 469--493. Springer, Cham, 2021{\natexlab{b}}.
	
	\bibitem[Falconer et~al.(2015)Falconer, Fraser, and Jin]{FalconerEtAl2015}
	K.~J. Falconer, J.~M. Fraser, and X.~Jin.
	\newblock Sixty years of fractal projections.
	\newblock In \emph{Fractal geometry and stochastics {V}}, volume~70 of
	\emph{Progr. Probab.}, pages 3--25. Birkh\"{a}user/Springer, Cham, 2015.
	
	\bibitem[Falconer et~al.(2020)Falconer, Fraser, and Kempton]{FalconerEtAl2020}
	K.~J. Falconer, J.~M. Fraser, and T.~Kempton.
	\newblock Intermediate dimensions.
	\newblock \emph{Math. Z.}, 296\penalty0 (1-2):\penalty0 813--830, 2020.
	
	\bibitem[Feng et~al.(2022)Feng, Lo, and Ma]{FengEtAl2022}
	D.-J. Feng, C.-H. Lo, and C.-Y. Ma.
	\newblock Dimensions of projected sets and measures on typical self-affine
	sets.
	\newblock {\em arXiv preprint
		\href{http://arxiv.org/abs/2209.00228}{arXiv:2209.00228}}, 2022.
	
	\bibitem[Fuglede(1960)]{Fuglede1960}
	B.~Fuglede.
	\newblock On the theory of potentials in locally compact spaces.
	\newblock \emph{Acta Math.}, 103:\penalty0 139--215, 1960.
	
	\bibitem[Hutchinson(1981)]{Hutchinson1981}
	J.~E. Hutchinson.
	\newblock Fractals and self-similarity.
	\newblock \emph{Indiana Univ. Math. J.}, 30\penalty0 (5):\penalty0 713--747,
	1981.
	
	\bibitem[J\"{a}rvenp\"{a}\"{a} et~al.(2014)J\"{a}rvenp\"{a}\"{a},
	J\"{a}rvenp\"{a}\"{a}, K\"{a}enm\"{a}ki, Koivusalo, Stenflo, and
	Suomala]{JaervenpaeaeEtAl2014}
	E.~J\"{a}rvenp\"{a}\"{a}, M.~J\"{a}rvenp\"{a}\"{a}, A.~K\"{a}enm\"{a}ki,
	H.~Koivusalo, O.~Stenflo, and V.~Suomala.
	\newblock Dimensions of random affine code tree fractals.
	\newblock \emph{Ergodic Theory Dynam. Systems}, 34\penalty0 (3):\penalty0
	854--875, 2014.
	
	\bibitem[Jordan et~al.(2007)Jordan, Pollicott, and Simon]{JordanEtAl2007}
	T.~Jordan, M.~Pollicott, and K.~Simon.
	\newblock Hausdorff dimension for randomly perturbed self affine attractors.
	\newblock \emph{Comm. Math. Phys.}, 270\penalty0 (2):\penalty0 519--544, 2007.
	
	\bibitem[K\"{a}enm\"{a}ki(2004)]{Kaeenmaeki2004}
	A.~K\"{a}enm\"{a}ki.
	\newblock On natural invariant measures on generalised iterated function
	systems.
	\newblock \emph{Ann. Acad. Sci. Fenn. Math.}, 29\penalty0 (2):\penalty0
	419--458, 2004.
	
	\bibitem[K\"{a}enm\"{a}ki and Vilppolainen(2010)]{KaeenmaekiVilppolainen2010}
	A.~K\"{a}enm\"{a}ki and M.~Vilppolainen.
	\newblock Dimension and measures on sub-self-affine sets.
	\newblock \emph{Monatsh. Math.}, 161\penalty0 (3):\penalty0 271--293, 2010.
	
	\bibitem[Kahane(1985)]{Kahane1985}
	J.-P. Kahane.
	\newblock \emph{Some random series of functions}, volume~5 of \emph{Cambridge
		Studies in Advanced Mathematics}.
	\newblock Cambridge University Press, Cambridge, second edition, 1985.
	
	\bibitem[Mattila(1995)]{Mattila1995}
	P.~Mattila.
	\newblock \emph{Geometry of sets and measures in {E}uclidean spaces: Fractals
		and rectifiability}, volume~44 of \emph{Cambridge Studies in Advanced
		Mathematics}.
	\newblock Cambridge University Press, Cambridge, 1995.
	
	\bibitem[Solomyak(1998)]{Solomyak1998}
	B.~Solomyak.
	\newblock Measure and dimension for some fractal families.
	\newblock \emph{Math. Proc. Cambridge Philos. Soc.}, 124\penalty0 (3):\penalty0
	531--546, 1998.
	
\end{thebibliography}
\end{document}